\pdfoutput=1

\documentclass[10pt]{article}

\usepackage[utf8]{inputenc}
\usepackage[T1]{fontenc}
\usepackage[english]{babel}

\usepackage[lmargin = 1.5 in, rmargin = 1 in, tmargin = 1 in, bmargin = 1 in]{geometry}

\usepackage[colorlinks]{hyperref}
\hypersetup{
  linkcolor=blue,
  urlcolor=blue
}%

\author{Charles Devlin VI\thanks{Dept. of Mathematics, University of Chicago, 5734 S University Ave, Chicago, IL, 60637, USA, \url{cpd6@uchicago.edu}}}

\usepackage[stretch=12]{microtype}

\usepackage{amsmath}
\usepackage{amsthm}
\usepackage{amssymb} 
\usepackage{mathtools}

\usepackage{color}

\usepackage{stackrel}

\usepackage{enumitem}

\usepackage{bm}

\allowdisplaybreaks 

\renewcommand{\AA}{\mathbb{A}}

\newcommand{\CC}{\mathbb{C}}

\newcommand{\NN}{\mathbb{N}}
\newcommand{\PP}{\mathbb{P}}

\newcommand{\RR}{\mathbb{R}}
\newcommand{\ZZ}{\mathbb{Z}}

\newcommand{\cD}{\mathcal{D}}
\newcommand{\cE}{\mathcal{E}}
\newcommand{\cL}{\mathcal{L}}
\newcommand{\cR}{\mathcal{R}}
\newcommand{\cS}{\mathcal{S}}
\newcommand{\cU}{\mathcal{U}}

\newcommand{\fa}{\mathfrak{a}}

\DeclareMathOperator{\Len}{Len}

\newcommand{\hlint}[2]{\left({#1}, {#2}\right]}
\newcommand{\hrint}[2]{\left[{#1}, {#2}\right)}

\newcommand{\D}{\mathrm{d}}

\newcommand{\w}{\wedge}

\renewcommand{\i}{{\mathrm{i}}}

\newtheorem{thm}{Theorem}[section]
\newtheorem{prop}[thm]{Proposition}
\newtheorem{lemma}[thm]{Lemma}

\theoremstyle{definition}
\newtheorem{defn}[thm]{Definition}
\newtheorem{rmk}[thm]{Remark}

\usepackage{tikz}

\newcommand{\ratios}[1][\epsilon]{\frac{r \fa_{{#1}}^{-1}}{r^{\xi Q} \fa_{{#1}/r}^{-1}}}
\newcommand{\invratios}[1][\epsilon]{\frac{r^{\xi Q} \fa_{{#1}/r}^{-1}}{r \fa_{{#1}}^{-1}}}

\NewDocumentCommand \locLFPP { o o } {%
  \IfNoValueTF{#2}
    {%
      \IfNoValueTF{#1}
        {%
          \fa_{\epsilon}^{-1} \hat{D}_h^{\epsilon}%
        }
        {%
          \fa_{#1}^{-1} \hat{D}_h^{#1}%
        }
    }
    {%
      \fa_{#1}^{-1} \hat{D}_{#2}^{#1}%
    }
}

\NewDocumentCommand \LFPP { o o } {%
  \IfNoValueTF{#2}
    {%
      \IfNoValueTF{#1}
        {%
          \fa_{\epsilon}^{-1} D_h^{\epsilon}%
        }
        {%
          \fa_{#1}^{-1} D_h^{#1}%
        }
    }
    {%
      \fa_{#1}^{-1} D_{#2}^{#1}%
    }
}

\title{Almost Sure Convergence of Liouville First Passage Percolation}
\date{}

\usepackage[style=alphabetic,maxnames=50]{biblatex}
\begin{document}
  \maketitle

  \abstract{\textit{Liouville first passage percolation (LFPP)} with parameter $\xi > 0$ is the family of random distance functions (metrics) $(D_h^{\epsilon})_{\epsilon > 0}$ on $\CC$ obtained heuristically by integrating $e^{\xi h}$ along paths, where $h$ is a variant of the Gaussian free field.
  There is a critical value $\xi_{\text{crit}} \approx 0.41$ such that for $\xi \in (0, \xi_{\text{crit}})$, appropriately rescaled LFPP converges in probability uniformly on compact subsets of $\CC$ to a limiting metric $D_h$ on $\gamma$-Liouville quantum gravity with $\gamma = \gamma(\xi) \in (0,2)$.
  We show that the convergence is almost sure, giving an affirmative answer to a question posed by Gwynne and Miller (2019).
  }

  \tableofcontents

  \section{Introduction}
    A \textit{Liouville quantum gravity} (LQG) surface is a random $2$-dimensional Riemannian manifold with Riemannian metric $e^{\gamma h}(\D x^2 + \D y^2)$, where $\gamma \in (0,2)$ is a parameter, where $h$ is a version of the Gaussian free field (GFF), and where $\D x^2 + \D y^2$ is the Euclidean metric tensor.
    The motivation is that $2$-dimensional Riemannian manifolds admit isothermal coordinates in a neighbourhood of each point, and taking $h$ to be a ``Gaussian random function'' should give rise to a ``natural'' model of random surfaces.
    This definition does not make sense without regularization since the GFF is only a random generalized function (distribution).
    We will give a brief exposition about how to make rigorous sense of this definition below, but see \cite{GwynneLQGSurvey}, \cite{ScottICM}, and \cite{BerestyckiPowell} for a more detailed introduction.

    We will consider the case that our surface has the topology of the plane $\CC$.
    To make the heuristic definition more precise, we define a random area measure and random distance function (metric) on $\CC$.
    Let $p_t(z) \coloneqq \frac{1}{2\pi t} e^{-|z|^2/2t}$.
    A random distribution $h$ on $\CC$ is a \textit{whole-plane GFF plus a (bounded) continuous function} if there is a coupling of $h$ with a random (bounded) continuous function $f \colon \CC \to \RR$ such that $h - f$ has the law of a whole-plane GFF.
    Given a whole-plane GFF plus a bounded continuous function, define an area measure as the almost sure weak limit \cite{RhodesVargasGMC} \cite{BerestyckiGMC} 
    \begin{align}
      \mu \
      &\coloneqq \ \lim_{\epsilon \to 0} \epsilon^{\gamma^2/2} e^{\gamma h_{\epsilon}^{*}(z)} \, \D z,
      \label{eq:LQGAreaMeasure}
    \end{align}
    where $\D z$ is Lebesgue measure, and where
    \begin{align*}
      h_{\epsilon}^{*}(z) \
      &\coloneqq \ (h \ast p_{\epsilon^2/2})(z) \
      = \ \int\limits_{\CC} h(w) p_{\epsilon^2/2}(z - w) \, \D w,
    \end{align*}
    where the integral is interpreted in the sense of distributional pairing.

    Keeping with the heuristic definition, the distance $D_h(z,w)$ between two points should be defined as the infimum over lengths of paths between $z$ and $w$, where paths are weighted by $e^{\xi h}$.
    Here, $\xi = \xi(\gamma) \coloneqq \gamma/d_{\gamma}$ where $d_{\gamma} > 0$ is the Hausdorff dimension of the LQG metric\footnote{The more precise definition of $d_{\gamma}$ is found in \cite{FractalDimensionOfLQG}. It is shown to agree with the Hausdorff dimension of the associated $\gamma$-LQG metric in \cite{KPZFormulasForLQG}.}.
    More precisely, define the $\epsilon$-\textit{Liouville first passage percolation} (LFPP) metric with parameter $\xi$ by
    \begin{align*}
      D_h^{\epsilon}(z,w) \
      &\coloneqq \ \inf_{P \colon z \to w} \int\limits_{0}^{1} e^{\xi h_{\epsilon}^{*}(P(t))} |P'(t)| \, \D t,
    \end{align*}
    where the infimum is over all piecewise differentiable paths $P \colon [0,1] \to \CC$ from $z$ to $w$.
    To obtain a nontrivial limit as $\epsilon \to 0$, we renormalize by
    \begin{align*}
      \fa_{\epsilon} \
      &\coloneqq \ \text{median of } \inf\left\{\int_{0}^{1} e^{\xi h_{\epsilon}^{*}(P(t))} |P'(t)| \right\}.
    \end{align*}
    where the infimum is over all piecewise differentiable paths $P \colon [0,1] \to [0,1]^2$ with $P(0) \in \{0\} \times [0,1]$ and $P(1) \in \{1\} \times [0,1]$.
    The exact value of $\fa_{\epsilon}$ is not known, but it is shown in \cite[Proposition 1.1]{TightnessOfSupercriticalLFPP} that 
    \begin{align*}
      \fa_{\epsilon} \
      &= \ \epsilon^{1 - \xi Q + o_{\epsilon}(1)} \text{ as } \epsilon \to 0
    \end{align*}
    for some (nonexplicit) exponent $Q = Q(\xi) > 0$.
    The $\gamma$-LQG metric is then defined by the limit
    \begin{align}
      D_h(z,w) \
      &\coloneqq \ \lim_{\epsilon \to 0} \LFPP(z,w).
      \label{eq:LQGDefinition}
    \end{align}

    We remark that this definition makes sense for arbitrary $\xi > 0$, but a phase transition occurs at a certain value $\xi = \xi_{\text{crit}}$ (see \cite{TightnessOfSupercriticalLFPP} and \cite{ExistenceAndUniqueness}).
    When $\xi \in \hlint{0}{\xi_{\text{crit}}}$, there is a corresponding $\gamma \in \hlint{0}{2}$ such that $\xi = \gamma/d_{\gamma}$, but when $\xi > \xi_{\text{crit}}$, the corresponding $\gamma$ is complex.
    The subcritical regime $\xi \in (0, \xi_{\text{crit}})$ will be the main focus of this paper.

    To make sense of the limit \eqref{eq:LQGDefinition} in the subcritical regime, we equip the space of continuous functions $\CC \times \CC \to \RR$ with the topology of uniform convergence on compact sets, then the convergence $\LFPP \to D_h$ is known in probability.
    The limit also exists in probability when $\xi \geq \xi_{\text{crit}}$, but one needs to work with a topology on the space of lower semicontinuous functions introduced in \cite{BeerTopology}. 
    The proof of the limit \eqref{eq:LQGDefinition} when $\xi \geq \xi_{\text{crit}}$ was carried out in \cite{UniquenessOfSupercriticalLQGMetrics}, building on \cite{TightnessOfSupercriticalLFPP} and \cite{WeakSupercriticalLQGMetrics}.

    Let us now breifly outline how the limit \eqref{eq:LQGDefinition} was shown in the subcritical phase $\xi \in (0, \xi_{\text{crit}})$.
    In \cite{TightnessOfSubcriticalLFPP}, the authors show that the laws of $\LFPP$ are tight with respect to the topology of local uniform convergence.
    A list of axioms which one should expect ``the Riemannian distance function associated to $e^{\gamma h}(\D x^2 + \D y^2)$'' to satisfy are given in \cite{ExistenceAndUniqueness} and \cite{WeakLQGMetrics} (we will state these axioms below in Section \ref{subsection:AxiomsOfLQGMetrics}).
    It is shown in \cite{WeakLQGMetrics} that subsequential limits of $\LFPP$ satisfy these axioms; in particular, using \cite[Corollary 1.8]{LocalMetricsOfTheGFF}, the authors verify that for any subsequential limit $(h, D)$ of $(h, \LFPP)$, the metric $D$ is a measurable function of $h$.
    By an elementary probabilistic lemma (see \cite[Lemma 4.5]{ContourLineOfContinuumGFF}), since the limiting metric is a measurable function of $h$, it follows that $\LFPP$ admits subsequential limits in probability, not just in law.
    In \cite{ExistenceAndUniqueness}, the axioms are shown to uniquely characterize the metric up to multiplicative constant.
    It follows that subsequential limits of $\LFPP$ are unique, and hence the limit \eqref{eq:LQGDefinition} exists in probability along the continuum index $\epsilon$.
    In addition to \cite{TightnessOfSubcriticalLFPP}, \cite{LocalMetricsOfTheGFF}, and \cite{WeakLQGMetrics}, the proof of uniqueness in \cite{ExistenceAndUniqueness} relies on the results of \cite{ConfluenceOfGeodesicsSubcriticalLQG}.

    This approach to showing convergence in probability has the unfortunate consequence of not obtaining quantitative estimates for the rate of convergence, so almost sure convergence doesn't easily follow.
    In contrast, the almost sure limit \eqref{eq:LQGAreaMeasure} has several known proofs in the case of the dyadic sequence $\epsilon = 2^{-n}$ \cite{Kahane} \cite{RhodesVargasGMC} \cite{DuplantierSheffield} \cite{Shamov}, and along the continuum index in \cite{SheffieldWangConformalCoordinateChange} when one replaces $h_{\epsilon}^{*}(z)$ by the circle average process $h_{\epsilon}(z)$.
    Determining whether the limit in \eqref{eq:LQGDefinition} is almost sure is Problem 7.7 in the list of open problems from \cite{ExistenceAndUniqueness}, and the main result of our paper is to give an affirmative answer.

    \begin{thm}
      \label{thm:AlmostSureConvergence}
      Almost surely, $\LFPP \to D_h$ uniformly on compacts.
    \end{thm}

    One important corollary of Theorem \ref{thm:AlmostSureConvergence} is Theorem \ref{thm:LQGScaling} below, which says that the spatial scale invariance property of the LQG metric (see axiom \ref{axiom:CoordinateChange} below) holds for all scalings simultaneously.
    This allows one to consider random affine coordinate changes, something which otherwise requires ad hoc arguments such as those of Section 2.4.1 of \cite{MinkowskiContentOfLQGMetric} to deal with.
    The analogous spatial scaling property for the area measure \eqref{eq:LQGAreaMeasure} was previously shown in \cite{SheffieldWangConformalCoordinateChange} (actually, \cite{SheffieldWangConformalCoordinateChange} proves the stronger result that the coordinate change formula for the area measure on subdomains of $\CC$ holds for all conformal transformations simultaneously).

    \begin{rmk}
      It is worth noting that the reason to define LFPP using $h_{\epsilon}^{*}(z)$ instead of some other continuous approximation of $h$ is that this is the version for which tightness of $\LFPP$ was shown in \cite{TightnessOfSubcriticalLFPP}.
      As noted in Remark 1.1 of \cite{WeakLQGMetrics}, if tightness is shown for a different approximation, then the arguments of \cite{WeakLQGMetrics} should apply to said approximation as well, and the same is true for the arguments in the present work.
    \end{rmk}

    \vspace{12pt} \noindent \textbf{Acknowledgements} We thank Jian Ding, Ewain Gwynne, and Jinwoo Sung for helpful discussions. 
    The author was partially supported by NSF grant DMS-2245832.

    \subsection{Axioms of LQG Metrics}
      \label{subsection:AxiomsOfLQGMetrics}
      A natural question is whether there could be multiple different metrics which match the heuristic definition of the LQG metric as the ``Riemannian distance function associated to $e^{\gamma h}(\D x^2 + \D y^2)$''.
      The authors of \cite{ExistenceAndUniqueness} address this question by stating a list of axioms which one should expect any reasonable notion of a $\gamma$-LQG metric to satisfy, then show that these axioms uniquely characterize the $\gamma$-LQG metric up to multiplicative constant.
      Moreover, the limit \eqref{eq:LQGDefinition} satisfies these axioms, making it correct to call this limit \textit{the} $\gamma$-LQG metric.
      As these axioms will be used throughout this paper, we will state them after recalling some terminology from metric space theory.

      Let $(X, D)$ be a metric space.
      If $A,B \subset X$, define
      \begin{align*}
        D(A,B) \
        &\coloneqq \ \inf_{\substack{x \in A \\ y \in B}} D(x,y).
      \end{align*}

      A \textit{curve} in $X$ is a continuous function $P \colon [a,b] \to X$, and its $D$-length is 
      \begin{align*}
        \Len(P; D) \
        &\coloneqq \ \sup_{T} \sum_{i=1}^{\# T} D(P(t_i), P(t_{i-1})),
      \end{align*}
      where the supremum is over all finite partitions $T = \{a = t_0 < t_1 < \cdots < t_{\# T} = b\}$ of $[a,b]$.

      If $Y \subset X$, the \textit{internal metric} of $D$ on $Y$ is
      \begin{align*}
        D(x,y; Y) \
        &\coloneqq \ \inf_{P} \Len(P; D), \ \forall x,y \in Y,
      \end{align*}
      where the infimum is over all curves in $Y$ from $x$ to $y$.
      Then $D(\cdot, \cdot; Y)$ is a metric on $Y$, but is allowed to take the value $\infty$.

      We say $(X, D)$ is a \textit{length space} if for each $x,y \in X$ and each $\epsilon > 0$, there is a curve from $x$ to $y$ with $D$-length at most $D(x,y) + \epsilon$.

      We will call a metric $D$ on an open subset $U \subset \CC$ \textit{continuous} if the identity mapping $(\CC, |\cdot|) \to (\CC, D)$ is a homeomorphism.
      Equip the space of continuous metrics on $U$ with the topology of local uniform convergence on functions $U \times U \to \hrint{0}{\infty}$ with its associated Borel $\sigma$-algebra.

      In the case $U$ is disconnected, we allow $D(x,y) = \infty$ when $x$ and $y$ are in different connected components of $U$.
      If $D^n$ is a sequence of continuous metrics on $U$, then to have $D^n \to D$ locally uniformly, we additionally require that for all $n$ sufficiently large, $D^n(x,y) = \infty$ if and only if $D(x,y) = \infty$.

      Equip the space $\cD'(\CC)$ of distributions on $\CC$ with the weak topology.
      If $\gamma \in (0, 2)$, a $\gamma$\textit{-Liouville quantum gravity metric} is a measurable function $h \mapsto D_h$ from $\cD'(\CC)$ to the space of continuous metrics on $\CC$ satisfying the following whenever $h$ is a whole-plane GFF plus a continuous function.
  
      \begin{enumerate}[label=\Roman*]
        \item \label{axiom:LengthSpace} \textbf{Length space.} Almost surely, $(\CC, D_h)$ is a length space.
        \item \label{axiom:Locality} \textbf{Locality.} If $U \subset \CC$ is a deterministic open set, then the internal metric $D_h(\cdot, \cdot; U)$ is almost surely determined by $h|_U$.
        \item \label{axiom:WeylScaling} \textbf{Weyl scaling.} For each continuous function $f \colon \CC \to \RR$, define
          \begin{align*}
            \left(e^{\xi f} \cdot D_h \right)(z,w) \
            &\coloneqq \ \inf_{P \colon z \to w} \int\limits_{0}^{\Len(P; D_h)} e^{\xi f(P(t))} \, \D t, \ \forall z,w \in \CC,
          \end{align*}
          where the infimum is over all curves $P$ from $z$ to $w$ parameterized by $D_h$-length.
          Then almost surely, $e^{\xi f} \cdot D_h = D_{h + f}$ for all continuous functions $f \colon \CC \to \RR$.
        \item \label{axiom:CoordinateChange} \textbf{Coordinate change for translation and scaling.} For each deterministic $a \in \CC \setminus \{0\}$and $b \in \CC$, almost surely
          \begin{align*}
            D_h(az + b, aw + b) \
            &= \ D_{h(a \cdot + b) + Q \log |a|}(z,w) \ \forall z,w \in \CC.
          \end{align*}
      \end{enumerate}

      For a discussion regarding why these axioms are natural, see \cite[Section 1.2]{ExistenceAndUniqueness}.
      It is worth emphasizing that axiom \ref{axiom:WeylScaling} holds for \textit{all} continuous functions simultaneously almost surely, whereas axiom \ref{axiom:CoordinateChange} only holds for each fixed $a$ and $b$ almost surely.

    \subsection{Outline of Main Results}
      Unless stated otherwise, we will assume $\xi \in (0, \xi_{\text{crit}})$.
      Our main result is Theorem \ref{thm:AlmostSureConvergence}.
      The proof relies on the main result of \cite{UpToConstants}, which says, roughly, that $\LFPP$ and $D_h$ are Lipschitz equivalent when one considers distances between Euclidean balls instead of points.
      For a more precise statement, see Proposition \ref{prop:UpToConstants} below.
      Our proof will show that the Lipschitz constant can be taken to be $1 + \delta$ with high probability for any $\delta > 0$.
      Then Borel-Cantelli and elementary estimates for $h_{\epsilon}^{*}$ will imply almost sure convergence along the sequence $\epsilon = n^{-a}$ for $a > 0$ large enough, and continuity estimates for $h_{\epsilon}^{*}$ will yield the same convergence along the continuum index $\epsilon$.

      The strategy to prove the Lipschitz constant can be made close to $1$ is inspired by the argument in Section 3.2 of \cite{ConformalCovariance}.
      The idea is that if we are given a $D_h$-geodesic $P$ and some $\delta > 0$, then the convergence in probability of $\LFPP$ to $D_h$ together with a local independence property for events determined by the GFF on disjoint concentric annuli (see Lemma \ref{lemma:IndependenceAcrossConcentricAnnuli} below) imply there are many ``good'' times $s$ and $t$ such that $\LFPP(P(s), P(t)) \leq C(\epsilon) D_h(P(s), P(t))$, where $C(\epsilon)$ is $\epsilon$-dependent, but should be thought of as close to $1 + \delta$ when $\epsilon$ is very small.
      More precisely, we will show that the segments of $P$ between these ``good'' times make up a positive proportion $\frac{1}{A+1} \in (0,1)$ of $P$.
      Here, $A > 0$ is a universal constant.
      From this, it follows that if $\LFPP$ and $D_h$ satisfy a Lipschitz condition with some Lipschitz constant $C_0$, then they also satisfy a Lipschitz condition with an $\epsilon$-dependent Lipschitz constant $C_1(\epsilon) \coloneqq \frac{A}{A+1} C_0 + \frac{1}{A+1} C(\epsilon)$ (we will actually allow $C_0$ to depend on $\epsilon$ as well).
      Since $A$ is universal, we can iterate this argument, taking $C_n(\epsilon)$ in place of $C_0$, to obtain a Lipschitz constant $C_{n+1}(\epsilon)$ for each $n \geq 0$.
      By taking $n$ large enough, we can make $C_n(\epsilon)$ as close to $C(\epsilon)$ as we wish.

      It remains to make precise the idea that $C(\epsilon)$ is close to $1 + \delta$ when $\epsilon$ is small.
      The actual definition of $C(\epsilon)$ is $1 + \delta$ times a term involving ratios of the form $\ratios$, which arise from the spatial scaling property of LFPP.
      It is known that for each fixed $r > 0$, $\lim_{\epsilon \to 0} \ratios = 1$ \cite[Corollary 1.11]{ExistenceAndUniqueness}, however this result is insufficient for our purposes since we will end up taking $r = r(\epsilon) \approx \epsilon^{1 - \zeta}$ for some $\zeta \in (0,1)$.
      So we need to argue that we can choose ``good'' values of $r \approx \epsilon^{1 - \zeta}$ for which $\ratios$ is close to $1$ for all $\epsilon$ small enough (see Lemma \ref{lemma:GoodRatios} below for a precise statement).

      A corollary of the existence of many ``good'' values of $r$ for which $\ratios$ is close to $1$ is a stronger estimate for the renormalization constants $\fa_{\epsilon}$ than the current state of the art. 

      \begin{thm}
        \label{thm:LFPPScalingConstants}
        For each $\xi \in (0, \xi_{\mathrm{crit}})$ and each $b > 0$, there is a constant $C = C(\xi, b) > 0$ such that for all $\epsilon \in (0, e^{-1})$,
        \begin{align}
          C^{-1} \epsilon^{1 - \xi Q} \left(\log \epsilon^{-1} \right)^{-b} \
          &\leq \ \fa_{\epsilon} \
          \leq \ C \epsilon^{1 - \xi Q} \left(\log \epsilon^{-1} \right)^b.
          \label{eq:LFPPScalingConstantsGoal}
        \end{align}
      \end{thm}

      The proof of Theorem \ref{thm:LFPPScalingConstants} is similar to the proof of \cite[Theorem 1.11]{UpToConstants}, which says the same thing as Theorem \ref{thm:LFPPScalingConstants}, but with $b = b(\xi)$ depending on $\xi$.
      The main difference is that Lemma \ref{lemma:GoodRatios} will allow us to replace the constant $A$ in the proof of \cite[Theorem 1.11]{UpToConstants} with $1 + \delta$ for any fixed $\delta > 0$.
      Since $b$ in \cite[Theorem 1.11]{UpToConstants} is roughly $\log A$, this will prove Theorem \ref{thm:LFPPScalingConstants}.

      As a corollary of Theorem \ref{thm:AlmostSureConvergence}, we will show that axiom \ref{axiom:CoordinateChange} holds for all scalings and translations simultaneously.

      \begin{thm}
        \label{thm:LQGScaling}
        There is a version of $(D_{h(a \cdot + b) + Q \log|a|})_{a \in \CC \setminus \{0\}, b \in \CC}$ such that almost surely, for all $a \in \CC \setminus \{0\}$ and all $b \in \CC$,
        \begin{align*}
          D_h\left(a \cdot + b, a \cdot + b \right) \
          &= \ D_{h(a \cdot + b) + Q \log|a|}(\cdot, \cdot).
        \end{align*}
      \end{thm}

      We emphasize that the difference between Theorem \ref{thm:LQGScaling} and axiom \ref{axiom:CoordinateChange} is that the order of quantifiers is swapped.
      That is, axiom \ref{axiom:CoordinateChange} holds for each fixed deterministic $a$ and $b$, but Theorem \ref{thm:LQGScaling} holds for \textit{all} $a$ and $b$ simultaneously.
      The proof is a consequence of the exact scale invariance of LFPP.

  \section{Preliminaries}
    \subsection{Notation}
      If $z \in \CC$ and $r > 0$, then $B_r(z) \coloneqq \{w \in \CC: |z - w| < r\}$.

      \noindent More generally, if $U \subset \CC$ and $r > 0$, then $B_r(U) \coloneqq \bigcup_{z \in U} B_{r}(z)$.

      \noindent If $z \in \CC$ and $0 < r_1 < r_2$, then $\AA_{r_1, r_2}(z) \coloneqq B_{r_2}(z) \setminus \overline{B_{r_1}(z)}$.

      \noindent If $f \colon (0, \infty) \to \RR$ and $g \colon (0,\infty) \to (0, \infty)$ are functions, we say $f(\epsilon) = O_{\epsilon}(g(\epsilon))$ (resp. $f(\epsilon) = o_{\epsilon}(g(\epsilon))$) if $f(\epsilon)/g(\epsilon)$ remains bounded (resp. converges to $0$) as $\epsilon \to 0$.
      Define $O_{\epsilon}$ and $o_{\epsilon}$ analogously when $\epsilon \to \infty$.

      \noindent If $D$ is a metric on an open subset $U \subset \CC$, and if $A \subset U$ is a region with the topology of an annulus, $D(\text{around } A)$ is the infimum over all $D$-lengths of curves in $A$ which disconnect the inner and outer boundaries.

    \subsection{Independence Across Concentric Annuli}
      The iterative argument used in the proof of Theorem \ref{thm:AlmostSureConvergence} involves comparing LFPP and LQG lengths of paths in small annuli.
      The following lemma, which is a special case of \cite[Lemma 3.1]{LocalMetricsOfTheGFF}, will be used to show that there are a large number of ``good annuli'' on which LFPP and LQG lengths are comparable.
      \begin{lemma}
        \label{lemma:IndependenceAcrossConcentricAnnuli}
        Fix $0 < s_1 < s_2 < 1$.
        Let $(r_k)_{k=1}^{\infty}$ be a decreasing sequence of positive numbers such that $r_{k+1}/r_k \leq s_1$ for each $k$, and let $(E_{r_k})_{k=1}^{\infty}$ be events such that $E_{r_k} \in \sigma\{(h - h_{r_k}(0))|_{\AA_{s_1 r_k, s_2 r_k}(0)}\}$ for each $k$.
        Then for each $a > 0$, there exists $p = p(a,s_1,s_2) \in (0,1)$ and $c = c(a,s_1,s_2) > 0$ such that if $\PP[E_{r_k}] \geq p$ for all $n \in \NN$, then
        \begin{align*}
          \PP\left\{E_{r_k} \text{ occurs for at least one } k \leq K \right\} \
          &\geq \ 1 - c e^{-a K} \quad \forall K \in \NN.
        \end{align*}
      \end{lemma}

    \subsection{Localized Approximation of LFPP}
      To apply Lemma \ref{lemma:IndependenceAcrossConcentricAnnuli}, the events $E_{r_k}$ need to depend only on the field $h$ restricted to small annuli.
      However, $h_{\epsilon}^{*}$ does not depend locally on the field due to the heat kernel being nonzero on all of $\CC$.
      So instead, we use a truncation $\hat{h}_{\epsilon}^{*}$ of $h_{\epsilon}^{*}$ that does depend locally on the field.
      In particular, we will use the truncation introduced in \cite{UpToConstants}, which has a range of dependence smaller than $\epsilon^{1 - \zeta}$ for any $\zeta \in (0,1)$.

      \newcommand{\frepsilon}{\epsilon \log \epsilon^{-1}} 
      For each $\epsilon > 0$, choose a deterministic, smooth, radially symmetric bump function $\psi_{\epsilon} \colon \CC \to [0,1]$ which is identically equal to $1$ on $B_{\frac{1}{2} \frepsilon}(0)$ and vanishes outside $B_{\frepsilon}(0)$.
      We can and do choose $\psi_{\epsilon}$ so that $(z,\epsilon) \mapsto \psi_{\epsilon}(z)$ is smooth.
      Define
      \begin{align*}
        \hat{h}_{\epsilon}^{*}(z) \
        &\coloneqq \ Z_{\epsilon}^{-1} \int\limits_{\CC} \psi_{\epsilon}(z - w) h(w) p_{\epsilon^2/2}(z - w) \, \D w,
      \end{align*}
      where the integral is in the sense of distributions, and where 
      \begin{align*}
        Z_{\epsilon} \
        &\coloneqq \ \int\limits_{\CC} \psi_{\epsilon}(w) p_{\epsilon^2/2}(w) \, \D w.
      \end{align*}
      Define the \textit{localized} $\epsilon$\textit{-LFPP metric} with parameter $\xi$ by
      \begin{align*}
        \hat{D}_h^{\epsilon}\left(z,w \right) \
        &\coloneqq \ \inf_{P \colon z \to w} \int\limits_{0}^{1} e^{\xi \hat{h}_{\epsilon}^{*}(P(t))} |P'(t)| \, \D t,
      \end{align*}
      where the infimum is over all piecewise differentiable curves $P$ from $z$ to $w$.

      The following lemma summarizes the properties of $\hat{h}_{\epsilon}^{*}$ and $\hat{D}_h^{\epsilon}$ which will be relevant to the present work.
      \begin{lemma}
        \label{lemma:PropertiesOfLocalizedFieldAndLFPP}
        \leavevmode
        \begin{enumerate}
          \item \label{loc:AddRandomVariableToField} If $c \in \RR$ is a random variable, then $\widehat{(h + c)}_{\epsilon}^{*}(z) = \hat{h}_{\epsilon}^{*}(z) + c$ for all $z \in \CC$.
          \item \label{loc:Locality} $\hat{h}_{\epsilon}^{*}(z)$ is almost surely determined by $h|_{B_{\frepsilon}(z)}$.
            Consequently, for any deterministic open set $U \subset \CC$, the internal metric $\hat{D}_h^{\epsilon}(\cdot, \cdot; U)$ is almost surely determined by $h|_{B_{\frepsilon}(U)}$.
          \item \label{loc:Continuity} $\hat{h}_{\epsilon}^{*}(z)$ has a modification which is jointly continuous in $z$ and $\epsilon$.
            We will always assume we are working with such a modification.
          \item \label{loc:UniformComparison} Let $U \subset \CC$ be a connected, bounded, open set.
            Almost surely, 
            \begin{align*}
              \lim_{\epsilon \to 0} \sup_{z \in \overline{U}} \left|h_{\epsilon}^{*}(z) - \hat{h}_{\epsilon}^{*}(z) \right| \
              &= \ 0,
            \end{align*}
            and 
            \begin{align*}
              \lim_{\epsilon \to 0} \frac{\hat{D}_h^{\epsilon}(z,w;V)}{D_h^{\epsilon}(z,w;V)} \
              &= \ 1, \ \text{ uniformly over all } z,w \in V \text{ with } z \neq w \text{ and all connected } V \subset U.
            \end{align*}
          \item \label{loc:TranslationInvariance} If $b \in \CC$ and if $H(z) \coloneqq h(z + b)$, then
            \begin{align*}
              \locLFPP(z + b, z + w) \
              &= \ \locLFPP[\epsilon][H](z,w) \ \forall (z,w) \in \CC.
            \end{align*}
          \item \label{loc:ConvergenceInProbability} $\locLFPP \to D_h$ locally uniformly in probability.
        \end{enumerate}

        \begin{proof}
          \ref{loc:AddRandomVariableToField} follows from the normalization constant $Z_{\epsilon}$ in the definition of $\hat{h}_{\epsilon}^{*}(z)$.
          \ref{loc:Locality} follows from the fact that $\psi_{\epsilon}$ vanishes outside $B_{\frepsilon}(0)$.
          For \ref{loc:Continuity} and \ref{loc:UniformComparison}, see \cite[Lemma 2.2]{UpToConstants}.
          The proof of \ref{loc:TranslationInvariance} follows from a change of variables to show that $\hat{h}_{\epsilon}^{*}(z + b) = \hat{H}_{\epsilon}^{*}(z)$.
          To see \ref{loc:ConvergenceInProbability}, apply \ref{loc:TranslationInvariance} with an increasing sequence of connected bounded open sets whose union is $\CC$ to deduce $\locLFPP \to D_h$ from $\LFPP \to D_h$. 
        \end{proof}
      \end{lemma}

      As mentioned earlier, the starting point for the present work is the approximate Lipschitz equivalence between $\LFPP$ and $D_h$ proven in \cite{UpToConstants}.
      This result holds with $\locLFPP$ in place of $\LFPP$ as well (in fact, the proof in \cite{UpToConstants} deduces the result for $\LFPP$ from the result for $\locLFPP$), so since we will be working with $\locLFPP$ for many of our proofs, it is worth stating the Lipschitz equivalence for both $\LFPP$ and $\locLFPP$.
      For details, see Theorem 1.8 and Proposition 3.6 in \cite{UpToConstants}.

      \begin{prop}
        \label{prop:UpToConstants}
        For each $\zeta \in (0, 1)$ there exists $\beta > 0$ and $C_0 > 0$ depending only on $\zeta$ and the law of $D_h$ such that the following is true.
        Let $U \subset \CC$ be a deterministic, connected, bounded open set.
        With probability at least $1 - O_{\epsilon}(\epsilon^{\beta})$ as $\epsilon \to 0$, 
        \begin{align}
          \fa_{\epsilon}^{-1} \hat{D}_h^{\epsilon}\left(B_{\epsilon^{1 - \zeta}}(z), B_{\epsilon^{1 - \zeta}}(w); B_{\epsilon^{1 - \zeta}}(U) \right) \
          &\leq \ C_0 D_h\left(z,w; U \right) \quad \forall z,w \in U, \label{eq:UpToConstantsLocDh} \\
          D_h\left(B_{\epsilon^{1 - \zeta}}(z), B_{\epsilon^{1 - \zeta}}(w); B_{\epsilon^{1 - \zeta}}(U) \right) \
          &\leq \ C_0 \fa_{\epsilon}^{-1} \hat{D}_h^{\epsilon}\left(z,w; U \right) \quad \forall z,w \in U, \label{eq:UpToConstantsDhLoc} \\
          \LFPP\left(B_{\epsilon^{1 - \zeta}}(z), B_{\epsilon^{1 - \zeta}}(w); B_{\epsilon^{1 - \zeta}}(U) \right) \
          &\leq \ C_0 D_h\left(u,v; U \right) \quad \forall z,w \in U, \label{eq:UpToConstantsLFPPDh} \\
          D_h\left(B_{\epsilon^{1 - \zeta}}(z), B_{\epsilon^{1 - \zeta}}(w); B_{\epsilon^{1 - \zeta}}(U) \right) \
          &\leq \ C_0 \LFPP\left(z,w ; U \right) \ \forall z,w \in U, \label{eq:UpToConstantsDhLFPP}
        \end{align}
        where the big-$O$ constant depends on $U$.
      \end{prop}

      We remark that although the statement of Theorem 1.8 in \cite{UpToConstants} only says ``with probability tending to $1$ as $\epsilon \to 0$'', the proof shows that the assertion holds with probability $1 - O_{\epsilon}(\epsilon^{\beta})$ for some $\beta > 0$ (this is the case for Proposition 3.6, and Theorem 1.8 is deduced from the former by means of the almost sure uniform comparison of $\LFPP$ and $\locLFPP$).

  \section{Setup for Theorem \ref{thm:AlmostSureConvergence}}
    \label{section:SetupForMainTheorem}
    For each $R, \epsilon > 0$, define
    \begin{align*}
      \hypertarget{V}{V(R)} \
      &\coloneqq \ \left\{(z,w) \in \CC^2 : \exists D_h\text{-geodesic from } z \text{ to } w \text{ contained in } B_R(0) \right\}, \\
      \hypertarget{Vhat}{\hat{V}_{\epsilon}(R)} \
      &\coloneqq \ \left\{(z,w) \in \CC^2 : \exists \hat{D}_h^{\epsilon}\text{-geodesic from } z \text{ to } w \text{ contained in } B_R(0) \right\},
    \end{align*}
    The goal of this section is to prove the following.

    \newcommand{\g}{\frac{1}{2}} 

    \begin{prop}
      \label{prop:MainIterationArgument}
      There is a constant $A > 0$ such that the following is true.
      Assume there exist $0 < R < \infty$, $\zeta_{-} \in (0, 1)$, $\beta > 0$, and $C_0(\epsilon), C_0'(\epsilon) > 0$ such that with probability $1 - O_{\epsilon}(\epsilon^{\beta})$ as $\epsilon \to 0$,
      \begin{align}
        \locLFPP\left(B_{4 \epsilon^{1 - \zeta_{-}}}(u), B_{4 \epsilon^{1 - \zeta_{-}}}(v) \right) \
        &\leq \ C_0(\epsilon) D_h\left(u,v \right) \ \forall (u,v) \in \hyperlink{V}{V(R)}, \label{eq:InitialLipschitzConstant} \\
        D_h\left(B_{4 \epsilon^{1 - \zeta_{-}}}(u), B_{4 \epsilon^{1 - \zeta_{-}}}(v) \right) \
        &\leq \ C_0'(\epsilon) \locLFPP(u,v) \ \forall (u,v) \in \hyperlink{Vhat}{\hat{V}_{\epsilon}(R)}. \label{eq:InitialLipschitzConstantReversed}
      \end{align}
      Choose $\delta, \zeta, \zeta_{+} \in (0,1)$ such that $\zeta_{-} < \zeta < \zeta_{+}$ and $1 - \zeta_{+} < \g (1 - \zeta)$. 
      For each $\epsilon \in (0,1)$, choose $\cR_{\epsilon} \subset (\epsilon^{1 - \zeta}, \epsilon^{1 - \zeta_{+}}) \cap \{8^j\}_{j \in \ZZ}$ with $\# \cR_{\epsilon} > \frac{1}{3} \# ((\epsilon^{1 - \zeta}, \epsilon^{1 - \zeta_{+}}) \cap \{8^j\}_{j \in \ZZ})$.
      Then with probability $1 - O_{\epsilon}(\epsilon^{\beta \w 2(1 - \zeta)})$ as $\epsilon \to 0$,
      \begin{align}
        \locLFPP\left(B_{4 \epsilon^{1 - \zeta_{+}}}(u), B_{4 \epsilon^{1 - \zeta_{+}}}(v) \right) \
        &\leq \ C_1(\epsilon) D_h\left(u,v\right) \ \forall (u,v) \in \hyperlink{V}{V(R)}, \label{eq:NewLipschitzConstant} \\
        D_h\left(B_{4 \epsilon^{1 - \zeta_{+}}}(u), B_{4 \epsilon^{1 - \zeta_{+}}}(v) \right) \
        &\leq \ C_1'(\epsilon) \locLFPP(u,v) \ \forall (u,v) \in \hyperlink{Vhat}{\hat{V}_{\epsilon}(R)}. \label{eq:NewLipschitzConstantReversed}
      \end{align}
      where 
      \begin{align}
        C_1(\epsilon) \
        &\coloneqq \ \frac{A}{A+1} \left[\sup_{r \in \cR_{\epsilon}} \ratios \vee C_0(\epsilon) \right] + \left[\frac{1 + \delta}{A+1} + 2 \delta \right] \sup_{r \in \cR_{\epsilon}} \ratios, \label{eq:UpdatedLipschitzConstant} \\
        C_1'(\epsilon) \
        &\coloneqq \ \frac{A}{A+1} \left[\sup_{r \in \cR_{\epsilon}} \invratios \vee C_0'(\epsilon) \right] + \left[\frac{1 + \delta}{A+1} + 2 \delta \right] \sup_{r \in \cR_{\epsilon}} \invratios. \label{eq:UpdatedLipschitzConstantReversed}
      \end{align}
    \end{prop}

    Before proceeding, let us comment on how we will use Proposition \ref{prop:MainIterationArgument}.
    The key idea is that the constant $A$ in Proposition \ref{prop:MainIterationArgument} does not depend on any of the other parameters, and in particular doesn't depend on $\zeta_{-}$, $\zeta$, and $\zeta_{+}$.
    So after applying the proposition with some choice of $\zeta_{-}, \zeta, \zeta_{+}$, equations \eqref{eq:NewLipschitzConstant} and \eqref{eq:NewLipschitzConstantReversed} imply we can apply the proposition again, this time with $\zeta_{+}$ in place of $\zeta_{-}$ and $(C_1(\epsilon), C_1'(\epsilon))$ in place of $(C_0(\epsilon), C_0'(\epsilon))$.
    As long as the $\sup_{r \in \cR_{\epsilon}} \ratios$ and $\sup_{r \in \cR_{\epsilon}} \invratios$ terms can be made close to $1$, then the Lipschitz constants from iteratively applying Proposition \ref{prop:MainIterationArgument} in this manner will eventually become close to $1$.

    The proof of Proposition \ref{prop:MainIterationArgument} is similar to the argument given in Section 3.2 of \cite{ConformalCovariance}.
    The idea is to argue that for a given $D_h$-geodesic or $\locLFPP$-geodesic from $u$ to $v$, a positive proportion of it lies in annuli on which $\locLFPP$ and $D_h$ are Lipschitz equivalent with Lipschitz constant $(1 + \delta) \ratios$ or $(1 + \delta) \invratios$.
    To find such annuli, we will apply Lemma \ref{lemma:IndependenceAcrossConcentricAnnuli} to the following events.

    \begin{defn}
      \label{defn:AnnulusEvents}
      For each $z \in \CC$, $\epsilon, \zeta_{-}, \delta \in (0,1)$, $\alpha \in (7/8, 1)$, $A > 1$, and $r \in (\epsilon, 1)$, let $E_{r,\epsilon}(z) = E_{r,\epsilon}(z; \alpha, \delta, \zeta_{-}, A)$ be the interesection of the events $E_{r,\epsilon}^{(j)}(z)$ with $j \in \{1,2,3,4\}$ defined by
      \begin{enumerate}
        \item \label{conditions:GoodLipschitzConstant} $E_{r,\epsilon}^{(1)}(z) = E_{r,\epsilon}^{(1)}(z; \alpha, \delta)$: For each $u \in \partial B_{\alpha r}(z)$ and each $v \in \partial B_r(z)$ such that there is a $D_h$-geodesic (resp. $\hat{D}_h^{\epsilon}$-geodesic) from $u$ to $v$ contained in $\overline{\AA_{\alpha r, r}(z)}$, we have respectively 
          \begin{align*}
            \locLFPP(u, v) \
            &\leq \ \ratios (1 + \delta) D_h(u,v; \AA_{r/2, 2r}(z)), \\
            D_h(u,v) \
            &\leq \ \invratios (1 + \delta) \locLFPP(u,v; \AA_{r/2 + \frepsilon, 2r - \frepsilon}(z)).
          \end{align*}
        \item \label{conditions:Measurability} $E_{r,\epsilon}^{(2)}(z) = E_{r, \epsilon}^{(2)}(z; \alpha)$: If $u \in \partial B_{\alpha r}(z)$ and $v \in \partial B_r(z)$, then
          \begin{itemize}
            \item $D_h(u,v; \overline{\AA_{\alpha r, r}(z)}) = D_h(u,v; \AA_{r/2, 2r}(z))$ implies $v$ is contained in the open $\locLFPP$-geodesic ball of radius $\locLFPP(u; \partial \AA_{r/2 + \frepsilon, 2r - \frepsilon}(z))$ centered at $u$;
            \item $\locLFPP(u,v; \overline{\AA_{\alpha r, r}(z)}) = \locLFPP(u,v; \AA_{r/2 + \frepsilon, 2r - \frepsilon}(z))$ implies that $v$ is contained in the open $D_h$-geodesic ball with radius $D_h(u, \partial \AA_{r/2, 2r}(z))$ centered at $u$.
          \end{itemize}
        \item \label{conditions:Proportions} $E_{r,\epsilon}^{(3)}(z) = E_{r,\epsilon}^{(3)}(z; \alpha, A)$: 
          \begin{align*}
            D_h(\text{around } \AA_{\alpha r, r}(z)) \
            &< \ A D_h(\partial B_{\alpha r}(z), \partial B_r(z)), \\
            \locLFPP(\text{around } \AA_{\alpha r, r}(z)) \
            &< \ A \locLFPP(\partial B_{\alpha r}(z), \partial B_r(z)).
          \end{align*}
        \item \label{conditions:Error} $E_{r,\epsilon}^{(4)}(z) = E_{r,\epsilon}^{(4)}(z; \alpha, \delta, \zeta_{-})$: For all $u \in \overline{\AA_{\alpha r, r}(z)}$ and $v \in \overline{B_{4 \epsilon^{1 - \zeta_{-}}}(u)}$, we have
          \begin{align*}
            \locLFPP\left(u,v; \AA_{r/2 + \frepsilon, 2r - \frepsilon}(z) \right) \
            &\leq \ \delta \ratios D_h\left(\partial B_{\alpha r}(z), \partial B_r(z) \right), \\
            D_h\left(u,v; \AA_{r/2, 2r}(z) \right) \
            &\leq \ \delta \invratios \locLFPP\left(\partial B_{\alpha r}(z), \partial B_r(z) \right).
          \end{align*}
      \end{enumerate}
    \end{defn}

    \begin{figure}[h]
      \centering
      \includegraphics[scale=0.6]{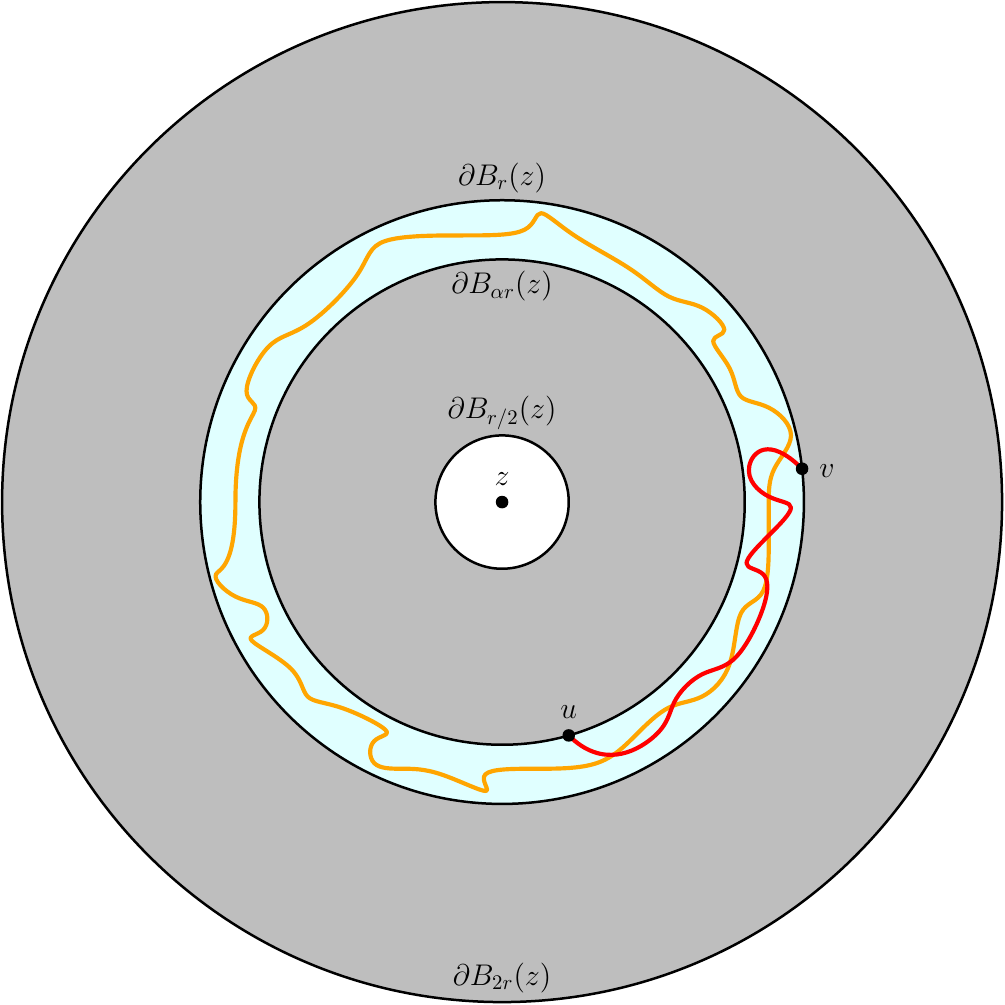}
      \caption{Illustration of Definition \ref{defn:AnnulusEvents}.
        Condition 1 says that if the red path is a $D_h$-geodesic, then $\locLFPP(u,v) \leq \ratios (1 + \delta) D_h(u,v; \AA_{r/2, 2r}(z))$, and analogously with $D_h$ and $\locLFPP$ swapped.
        Condition 2 says that if the $D_h$-distance from $u$ to $v$ is larger than the $D_h$-distance from $u$ to the boundary of the gray annulus, then the red path cannot be a $\locLFPP$-geodesic, and analogously with $D_h$ and $\locLFPP$ swapped.
        Condition 3 says there is a path (shown in orange) around the cyan annulus with $D_h$-length at most $A$ times the $D_h$-distance across said annulus, and analogously with $D_h$ and $\locLFPP$ swapped.
      }
    \end{figure}

    The strategy to prove Proposition \ref{prop:MainIterationArgument} will be to use Lemma \ref{lemma:IndependenceAcrossConcentricAnnuli} to prove that with high probability, for every point $z$ in a fine lattice, there is some radius $r$ such that $E_{r,\epsilon}(z)$ occurs.
    We will then break up a ($D_h$- or $\hat{D}_h^{\epsilon}$-) geodesic into segments which cross between the inner and outer boundaries of one of the annuli $\AA_{\alpha r, r}(z)$, and the segments between consecutive crossings.
    The former are ``good segments'' since condition \ref{conditions:GoodLipschitzConstant} in Definition \ref{defn:AnnulusEvents} says that the $D_h$ and $\locLFPP$ distances between the endpoints of these segments are comparable with Lipschitz constant $(1 + \delta) \ratios$ or $(1 + \delta) \invratios$.
    The segments between consecutive annulus crossings are ``bad segments'', and to bound these, we will use the Lipschitz constant $C_0(\epsilon)$ from \eqref{eq:InitialLipschitzConstant} or $C_0'(\epsilon)$ from \eqref{eq:InitialLipschitzConstantReversed}.
    Since \eqref{eq:InitialLipschitzConstant} and \eqref{eq:InitialLipschitzConstantReversed} only hold when the left-hand side is distances between balls rather than distances between points, condition \ref{conditions:Error} is needed to say that the error terms from estimating point-to-point distances using ball-to-ball distances is small.

    The point of condition \ref{conditions:Proportions} is to ensure that a positive proportion (roughly a $\frac{1}{A+1}$-proportion) of a given geodesic is comprised of the ``good segments''. 
    This is where the $\frac{1 + \delta}{A+1} \sup_{r \in \cR_{\epsilon}} \ratios$ and $\frac{1 + \delta}{A+1} \sup_{r \in \cR_{\epsilon}} \invratios$ come from in \eqref{eq:UpdatedLipschitzConstant} and \eqref{eq:UpdatedLipschitzConstantReversed}.
    The purpose of condition \ref{conditions:Measurability} is to ensure that $E_{r,\epsilon}(z)$ is determined by the field on $\AA_{r/2, 2r}(z)$, which is needed to apply Lemma \ref{lemma:IndependenceAcrossConcentricAnnuli}.

    Once we have proven Proposition \ref{prop:MainIterationArgument}, the idea will be to apply it repeatedly to an increasing sequence of $\zeta_{-}$'s, $\zeta$'s, and $\zeta_{+}$'s to gradually improve the Lipschitz constant.
    For this, it is \textit{crucial} that $A$ is a universal constant independent of all other parameters in the proposition; if $A$ were to depend on $\zeta_{-}$, $\zeta$, and/or $\zeta_{+}$, then the proportion of ``good segments'' could approach $0$ rapidly enough that the Lipschitz constants don't approach $1$ in the limit, and the remaining arguments of this paper completely fall apart.

    Finally, let us comment on why the ratios $\ratios$ and $\invratios$ appear throughout Proposition \ref{prop:MainIterationArgument} and Definition \ref{defn:AnnulusEvents}.
    The idea is that we need to choose the parameters in Definition \ref{defn:AnnulusEvents} to make $E_{r,\epsilon}(z)$ occur with high probability for each $\epsilon$ sufficiently small and for a large enough collection of $r$'s to apply Lemma \ref{lemma:IndependenceAcrossConcentricAnnuli}.
    We would like to choose said parameters for $r = 1$, then rescale space so the same parameters work for other values of $r$.
    Since $\locLFPP$ isn't scale invariant, we instead need to use \ref{loc:UniformComparison} from Lemma \ref{lemma:PropertiesOfLocalizedFieldAndLFPP} and exact scale invariance of $\LFPP$.
    Since this exact scale invariance is used repeatedly throughout the following arguments, it is worth stating.
    If $a \in \CC \setminus \{0\}$ and $z,w \in \CC$, then
    \begin{align}
      \LFPP\left(az, aw\right) \
      &= \ \frac{|a| \fa_{\epsilon}^{-1}}{|a|^{\xi Q} \fa_{\epsilon/|a|}^{-1}} \LFPP[\epsilon/|a|][h(a \cdot) + Q \log|a|](z,w). \label{eq:LFPPScaling}
    \end{align}

    \subsection{Choosing the Parameters in Definition \ref{defn:AnnulusEvents}}
      In this section, we use Lemma \ref{lemma:IndependenceAcrossConcentricAnnuli} to prove the following.
      \begin{lemma}
        \label{lemma:ChoosingParameters}
        There are universal constants $c, A > 0$ and $\alpha \in (7/8, 1)$ such that the following is true.
        Let $0 < \zeta_{-} < \zeta < \zeta_{+} < 1$ with $1 - \zeta_{+} < \g(1 - \zeta)$. 
        Let $\delta \in (0,1)$, and for each $\epsilon \in (0,1)$, let $\cR_{\epsilon} \subset (\epsilon^{1 - \zeta}, \epsilon^{1 - \zeta_{+}}) \cap \{8^j\}_{j \in \ZZ}$ with $\# \cR_{\epsilon} > \frac{1}{3} \# ((\epsilon^{1 - \zeta}, \epsilon^{1 - \zeta_{+}}) \cap \{8^j\}_{j \in \ZZ})$.
        There exists $\epsilon_0 \in (0,1)$ such that for each $z \in \CC$ and each $\epsilon \in (0, \epsilon_0)$,
        \begin{align*}
          P\left\{E_{r, \epsilon}(z) \text{ occurs for some } r \in \cR_{\epsilon} \right\} \
          &\geq \ 1 - c \epsilon^{4(1 - \zeta)}.
        \end{align*}
      \end{lemma}

      We will start by proving the events $E_{r,\epsilon}(z)$ are determined by the field restricted to a small annulus.

      \begin{lemma}
        \label{lemma:Measurability}
        For each $z \in \CC$, $\alpha \in (7/8, 1)$, $\epsilon \in (0,1)$, $r \in (4 \frepsilon + 16 \epsilon^{1 - \zeta_{-}}, \infty)$, $\alpha \in (7/8, 1)$, $A > 1$, $\delta \in (0,1)$, and $\zeta_{-} \in (0,1)$, 
        \begin{align*}
          E_{r,\epsilon}(z) \
          &\in \ \sigma\{(h - h_{4r}(z))|_{\AA_{r/2, 2r}(z)}\}.
        \end{align*}

        \begin{proof}
          This proof is the same as the proof of \cite[Lemma 3.8]{ConformalCovariance}, but we include it for completeness.
          The constraint $r \in (4 \frepsilon + 16 \epsilon^{1 - \zeta_{-}}, \infty)$ ensures each condition in Definition \ref{defn:AnnulusEvents} makes sense.
          Subtracting $h_{4r}(z)$ from $h$ scales both $D_h$ and $\locLFPP$ by $e^{-\xi h_{4r}(z)}$, which doesn't affect the occurrence of any of the conditions in Definition \ref{defn:AnnulusEvents}.
          So we may assume $h_{4r}(z) = 0$.

          $E_{r,\epsilon}^{(3)}(z)$ is determined by the internal metrics of $D_h$ and $\locLFPP$ on $\AA_{\alpha r,r}(z)$.
          Since $r > 4 \frepsilon$, axiom \ref{axiom:Locality} and \ref{loc:Locality} from Lemma \ref{lemma:PropertiesOfLocalizedFieldAndLFPP} imply these internal metrics are determined by $h|_{\AA_{r/2, 2r}(z)}$.

          The internal metrics of $D_h$ and $\locLFPP$ on $\overline{\AA_{\alpha r, r}(z)}$ and of $D_h$ on $\AA_{r/2, 2r}(z)$ and $\locLFPP$ on $\AA_{r/2 + \frepsilon, 2r - \frepsilon}(z)$ are measurable with respect to $h|_{\AA_{r/2, 2r}(z)}$ by axiom \ref{axiom:Locality} and \ref{loc:Locality} from Lemma \ref{lemma:PropertiesOfLocalizedFieldAndLFPP}.
          So for every $u \in \partial B_{\alpha r}(z)$, both $\locLFPP(u, \partial \AA_{r/2 + \frepsilon, 2r - \frepsilon}(z))$ and $D_h(u, \partial \AA_{r/2, 2r}(z))$, which are determined by said internal metrics, are also measurable with respect to $h|_{\AA_{r/2, 2r}(z)}$.
          Moreover, the open $D_h$-geodesic ball of radius $D_h(u, \partial \AA_{r/2, 2r}(z))$ centered at $u$ and the open $\locLFPP$-geodesic ball of radius $\locLFPP(u, \partial \AA_{r/2 + \frepsilon, 2r - \frepsilon}(z))$ centered at $u$ are, by definition, contained in $\AA_{r/2, 2r}(z)$ and $\AA_{r/2 + \frepsilon, 2r - \frepsilon}(z)$ respectively.
          It follows that $E_{r, \epsilon}^{(2)}(z)$ is measurable with respect to $h|_{\AA_{r/2, 2r}(z)}$.

          All of the quantities involved in the definition of $E_{r,\epsilon}^{(4)}$ are determined by the internal metrics of $D_h$ on $\AA_{r/2, 2r}(z)$ and of $\locLFPP$ on $\AA_{r/2 + \frepsilon, 2r - \frepsilon}(z)$, so $E_{r, \epsilon}^{(4)}(z) \in \sigma\{h|_{\AA_{r/2, 2r}(z)}\}$.

          It is not true that $E_{r,\epsilon}^{(1)}(z) \in \sigma\{h|_{\AA_{r/2, 2r}(z)}\}$, but we can still show that $E_{r,\epsilon}^{(1)}(z) \cap E_{r, \epsilon}^{(2)}(z) \in \sigma\{h|_{\AA_{r/2, 2r}(z)}\}$.
          Let $E_{r,\epsilon}^{(1')}(z)$ be the event that for each $u \in \partial B_{\alpha r}(z)$ and each $v \in \partial B_r(z)$ such that there is a $D_h(\cdot,\cdot; \AA_{r/2,2r}(z))$-geodesic (resp $\locLFPP(\cdot, \cdot; \AA_{r/2 + \frepsilon, 2r - \frepsilon}(z))$-geodesic) from $u$ to $v$ contained in $\overline{\AA_{\alpha r, r}(z)}$, we have respectively
          \begin{align*}
            &\locLFPP\left(u,v; \AA_{r/2 + \frepsilon, 2r - \frepsilon}(z) \right) \
            \leq \ \ratios (1 + \delta) D_h\left(u,v; \AA_{r/2, 2r}(z) \right), \\
            &D_h\left(u,v; \AA_{r/2, 2r}(z) \right) \
            \leq \ \invratios (1 + \delta) \locLFPP(u,v; \AA_{r/2 + \frepsilon, 2r - \frepsilon}(z)).
          \end{align*}
          We claim that $E_{r, \epsilon}^{(1)}(z) \cap E_{r,\epsilon}^{(2)}(z) = E_{r,\epsilon}^{(1')}(z) \cap E_{r,\epsilon}^{(2)}(z)$.

          Assume $E_{r,\epsilon}^{(1)}(z) \cap E_{r,\epsilon}^{(2)}(z)$ occurs, and assume $u \in \partial B_{\alpha r}(z)$ and $v \in \partial B_r(z)$ are such that there is a $D_h(\cdot, \cdot; \AA_{r/2, 2r}(z))$-geodesic from $u$ to $v$ contained in $\overline{\AA_{\alpha r, r}(z)}$.
          Then $D_h(u,v; \AA_{r/2, 2r}(z)) = D_h(u,v; \overline{\AA_{\alpha r, r}(z)})$, so the occurrence of $E_{r,\epsilon}^{(2)}(z)$ implies
          \begin{align*}
            \locLFPP(u,v) \
            &= \ \locLFPP(u,v; \AA_{r/2 + \frepsilon, 2r - \frepsilon}(z)).
          \end{align*}
          Then since $E_{r, \epsilon}^{(1)}(z)$ occurs,
          \begin{align*}
            \locLFPP\left(u, v; \AA_{r/2 + \frepsilon, 2r - \frepsilon}(z) \right)
            &= \locLFPP(u,v)
            \leq \ratios (1 + \delta) D_h(u,v; \AA_{r/2, 2r}(z)).
          \end{align*}
          Likewise, if there is a $\locLFPP(\cdot, \cdot; \AA_{r/2 + \frepsilon, 2r - \frepsilon}(z))$-geodesic from $u$ to $v$ contained in $\overline{\AA_{\alpha r, r}(z)}$, then $\locLFPP(u,v; \AA_{r/2 + \frepsilon, 2r - \frepsilon}(z)) = \locLFPP(u,v; \overline{\AA_{\alpha r, r}(z)})$, hence the occurrence of $E_{r,\epsilon}^{(2)}(z)$ implies $D_h(u,v) = D_h(u,v; \AA_{r/2, 2r}(z))$.
          Since $E_{r,\epsilon}^{(1)}(z)$ occurs,
          \begin{align*}
            D_h\left(u,v; \AA_{r/2, 2r}(z) \right) \
            &= \ D_h(u,v) \
            \leq \ \invratios (1 + \delta) \locLFPP\left(u,v; \AA_{r/2 + \frepsilon, 2r - \frepsilon}(z) \right)
          \end{align*}
          Thus, $E_{r, \epsilon}^{(1)}(z) \cap E_{r,\epsilon}^{(2)}(z) \subset E_{r,\epsilon}^{(1')}(z) \cap E_{r, \epsilon}^{(2)}(z)$.

          Conversely, assuming $E_{r,\epsilon}^{(1')}(z) \cap E_{r,\epsilon}^{(2)}(z)$ occurs and that $u \in \partial B_{\alpha r}(z)$ and $v \in \partial B_r(z)$ are such that there is a $D_h$-geodesic from $u$ to $v$ contained in $\overline{\AA_{\alpha r, r}(z)}$, then this geodesic is also a $D_h(\cdot, \cdot;\AA_{r/2, 2r}(z))$-geodesic, so since $E_{r,\epsilon}^{(1')}(z)$ occurs,
          \begin{align*}
            \locLFPP\left(u,v \right) \
            &\leq \ \locLFPP\left(u,v; \AA_{r/2 + \frepsilon, 2r - \frepsilon}(z) \right) \\
            &\leq \ \ratios (1 + \delta) D_h\left(u,v; \AA_{r/2, 2r}(z) \right).
          \end{align*}
          If instead there is a $\locLFPP$-geodesic from $u$ to $v$ contained in $\overline{\AA_{\alpha r, r}(z)}$, this geodesic is also a $\locLFPP(\cdot, \cdot; \AA_{r/2 + \frepsilon, 2r - \frepsilon}(z))$-geodesic, so assuming $E_{r,\epsilon}^{(1')}(z)$ occurs, we get
          \begin{align*}
            D_h\left(u,v \right) \
            &\leq \ D_h\left(u,v; \AA_{r/2, 2r}(z) \right) \
            \leq \ \invratios (1 + \delta) \locLFPP\left(u,v; \AA_{r/2 + \frepsilon, 2r - \frepsilon}(z) \right).
          \end{align*}

          Since $E_{r,\epsilon}^{(1')}(z)$ is determined by the internal metrics of $D_h$ on $\AA_{r/2, 2r}(z)$ and of $\locLFPP$ on $\AA_{r/2 + \frepsilon, 2r - \frepsilon}(z)$, we have $E_{r, \epsilon}^{(1')}(z) \in \sigma\{h|_{\AA_{r/2, 2r}(z)}\}$.
          It follows that $E_{r,\epsilon}^{(1)}(z) \cap E_{r, \epsilon}^{(2)}(z) = E_{r, \epsilon}^{(1')}(z) \cap E_{r,\epsilon}^{(2)}(z) \in \sigma\{h|_{\AA_{r/2, 2r}(z)}\}$.
        \end{proof}
      \end{lemma}

      \begin{lemma}
        \label{lemma:Condition1}
        Fix $\delta, p \in (0,1)$, $\alpha \in (7/8, 1)$, and $\zeta_{-} \in (0,1)$.
        There exists $\epsilon_0 = \epsilon_0(\delta, p, \alpha, \zeta_{-}) \in (0,1)$ such that for each $z \in \CC$, each $\epsilon \in (0, \epsilon_0)$, and each $r \in (\epsilon^{1 - \zeta_{-}}, 1)$, $E_{r,\epsilon}^{(1)}(z)$ occurs with probability at least $p$.

        \begin{proof}
          By \ref{loc:TranslationInvariance} in Lemma \ref{lemma:PropertiesOfLocalizedFieldAndLFPP}, axiom \ref{axiom:CoordinateChange}, and translation invariance of the law of a whole-plane GFF modulo additive constants, it will suffice to prove the claim when $z = 0$.
          Using the convergence in probability of $\LFPP$ and $\locLFPP$ to $D_h$, we can choose $R > 0$ and $\epsilon_0 \in (0,1)$ such that for each $\epsilon \in (0, \epsilon_0)$, it holds with probability at least $1 - (1 - p)/3$ that
          \begin{align}
            \LFPP\left(u,v \right) \
            &= \ \LFPP\left(u,v; B_R(0) \right) \ \forall u,v \in \overline{B_1(0)}, \label{eq:Condition1InternalMetricLFPP} \\
            \locLFPP\left(u,v \right) \
            &= \ \locLFPP\left(u,v; B_R(0) \right) \ \forall u,v \in \overline{B_1(0)}. \label{eq:Condition1InternalMetricLocalizedLFPP}
          \end{align}
          Fix $\delta' \in (0,1)$ with $1 + \delta' < \sqrt{1 + \delta}$.
          Apply \ref{loc:UniformComparison} from Lemma \ref{lemma:PropertiesOfLocalizedFieldAndLFPP} to decrease $\epsilon_0$ so that for each $\epsilon \in (0, \epsilon_0)$, it holds with probability at least $1 - (1-p)/3$ that
          \begin{align}
            \frac{1}{1 + \delta'} D_h^{\epsilon}(u,v; B_R(0)) \
            &\leq \ \hat{D}_h^{\epsilon}(u,v; B_R(0)) \
            \leq \ \left(1 + \delta' \right) D_h^{\epsilon}(u,v; B_R(0)) \ \forall u,v \in B_R(0) \label{eq:Condition1CompareLFPPAndLocalizedLFPP}
          \end{align}
          Use the convergence in probability of $\LFPP$ and $D_h$ to shrink $\epsilon_0$ further so that for each $\epsilon \in (0, \epsilon_0^{\zeta_{-}})$, it holds with probability at least $1 - (1-p)/3$ that
          \begin{align}
            \frac{1}{1 + \delta'} \LFPP\left(u,v \right) \
            &\leq \ D_h(z,w) \
            \leq \ (1 + \delta') \LFPP\left(u,v \right) \ \forall u \in \partial B_{\alpha}(0), v \in \partial B_1(0).
            \label{eq:Condition1CompareLFPPAndDh}
          \end{align}

          If $\epsilon \in (0, \epsilon_0)$ and  $r \in (\epsilon^{1 - \zeta_{-}}, 1)$, then with probability at least $p$, for all $u \in \partial B_{\alpha r}(0)$ and $v \in \partial B_r(0)$,
          \begin{align*}
            \locLFPP(u,v) \
            &\leq \ \left(1 + \delta' \right) \LFPP(u,v) & (\text{by \eqref{eq:Condition1InternalMetricLFPP}, \eqref{eq:Condition1InternalMetricLocalizedLFPP}, \eqref{eq:Condition1CompareLFPPAndLocalizedLFPP}}) \\
            &= \ \left(1 + \delta' \right) r^{\xi Q} e^{\xi h_r(0)} \ratios \LFPP[\epsilon/r][h(r \cdot) - h_r(0)](u/r, v/r) & (\text{by \eqref{eq:LFPPScaling}}) \\
            &\leq \ \left(1 + \delta \right) r^{\xi Q} e^{\xi h_r(0)} \ratios D_{h(r \cdot) - h_r(0)}(u/r, v/r) & (\text{by \eqref{eq:Condition1CompareLFPPAndDh}}) \\
            &= \ \left(1 + \delta \right) \ratios D_h(u,v) & (\text{by axioms \ref{axiom:WeylScaling} and \ref{axiom:CoordinateChange}}).
          \end{align*}
          Likewise,
          \begin{align*}
            D_h(u,v) \
            &= \ r^{\xi Q} e^{\xi h_r(0)} D_{h(r \cdot) - h_r(0)}(u/r, v/r) & (\text{by axioms \ref{axiom:WeylScaling} and \ref{axiom:CoordinateChange}}) \\
            &\leq \ (1 + \delta') r^{\xi Q} e^{\xi h_r(0)} \LFPP[\epsilon/r][h(r \cdot) - h_r(0)](u/r, v/r) & (\text{by \eqref{eq:Condition1CompareLFPPAndDh}}) \\
            &= \ (1 + \delta') \invratios \LFPP(u,v) & (\text{by \eqref{eq:LFPPScaling}}) \\
            &\leq \ (1 + \delta) \invratios \locLFPP(u,v) & (\text{by \eqref{eq:Condition1CompareLFPPAndLocalizedLFPP}}).
          \end{align*}
        \end{proof}
      \end{lemma}

      \begin{lemma}
        \label{lemma:Condition2}
        For each $p, \zeta \in (0,1)$, there exists $\alpha = \alpha(p) \in (7/8, 1)$ and $\epsilon_0 = \epsilon_0(\alpha, p, \zeta) \in (0,1)$ such that for each $z \in \CC$, each $\epsilon \in (0, \epsilon_0)$, and each $r \in (\epsilon^{1 - \zeta},1)$, $P[E_{r,\epsilon}^{(2)}(z)] \geq p$.
      \end{lemma}

      To prove Lemma \ref{lemma:Condition2}, we will need the following, which is an extension of \cite[Lemma 2.11]{ExistenceAndUniqueness}.

      \begin{lemma}
        \label{lemma:LowerBoundForDistancesInANarrowAnnulus}
        If $D$ and $\tilde{D}$ are metrics on $\CC$, $s > 0$, $\alpha \in (1/2, 1)$, $z \in \CC$, and $r \in \hlint{0}{1}$, define 
        \begin{align*}
          I(D, \tilde{D}, s, \alpha, r, z) \
          &\coloneqq \ \inf\left\{D(u,v; \overline{\AA_{\alpha r, r}(z)}) : u \in \partial B_{\alpha r}(z), v \in \partial B_r(z), \tilde{D}(u,v) > s \right\}.
        \end{align*}
        Then for each $0 < s < S$ and each $p, \zeta \in (0,1)$, there exists $\alpha_{*} = \alpha_{*}(s,S, p) \in (1/2, 1)$ such that for each $\alpha \in \hrint{\alpha_{*}}{1}$, there exists $\epsilon_0 = \epsilon_0(\alpha, \zeta, p) \in (0,1)$ such that for each $z \in \CC$, each $r \in (\epsilon^{1 - \zeta}, 1)$, and each $\epsilon \in (0, \epsilon_0)$, with probability at least $p$,
        \begin{align}
          I(\locLFPP, D_h, s r^{\xi Q} e^{\xi h_r(z)}, \alpha, r, z) \
          &\geq \ S \ratios r^{\xi Q} e^{\xi h_r(z)}, \label{eq:SecondClaim} \\
          I\left(D_h, \locLFPP, s \ratios r^{\xi Q} e^{\xi h_r(z)}, \alpha, r, z \right) \
          &\geq \ S r^{\xi Q} e^{\xi h_r(z)}. \label{eq:ThirdClaim}
        \end{align}

        \begin{proof}
          Lemma 2.11 from \cite{ExistenceAndUniqueness} guarantees the existence of $\alpha_{**} = \alpha_{**}(s, S, p) \in (1/2,1)$ such that for each $\alpha \in \hrint{\alpha_{**}}{1}$, each $z \in \CC$, and each $r \in \hlint{0}{1}$, it holds with probability at least $p$ that
          \begin{align}
            I(D_h, D_h, s r^{\xi Q} e^{\xi h_r(z)}, \alpha, r, z) \
            &\geq \ S r^{\xi Q} e^{\xi h_r(z)}.
            \label{eq:Condition2SublemmaPreliminaryStep}
          \end{align}
          We can then take $\alpha_{*}(s,S, p) \coloneqq \alpha_{**}(s/8, 8S, 1 - (1 - p)/3)$ (in particular, not depending on $\zeta$), and deduce \eqref{eq:SecondClaim} and \eqref{eq:ThirdClaim} by means of \ref{loc:UniformComparison} from Lemma \ref{lemma:PropertiesOfLocalizedFieldAndLFPP}, the convergence in probability of $\LFPP$ and $\locLFPP$ to $D_h$, and the exact scaling relation \eqref{eq:LFPPScaling} for LFPP.
          Use of \eqref{eq:LFPPScaling} is why the $\ratios$ terms in \eqref{eq:SecondClaim} and \eqref{eq:ThirdClaim} appear.
        \end{proof}
      \end{lemma}

      \begin{proof}[Proof of Lemma \ref{lemma:Condition2}]
        By translation invariance, it will suffice to prove the claim when $z = 0$.
        By continuity and scaling invariance, there exist $0 < s < S$ such that for each $r > 0$, it holds with probability at least $1 - (1 - p)/4$ that
        \begin{align*}
          D_h\left(\partial \AA_{3r/4, r}(0), \partial \AA_{r/2, 2r}(0) \right) \
          &\geq \ s r^{\xi Q} e^{\xi h_r(0)}, \\
          \sup_{u,v \in \overline{\AA_{3r/4, r}(0)}} D_h\left(u,v; \AA_{r/2, 2r}(0) \right) \
          &< \ S r^{\xi Q} e^{\xi h_r(0)}.
        \end{align*}
        Using Lemma \ref{lemma:PropertiesOfLocalizedFieldAndLFPP} and the definitions of $D_h^{\epsilon}$ and $\hat{D}_h^{\epsilon}$, we can find $\epsilon_0 \in (0,1)$ such that for each $\epsilon \in (0, \epsilon_0)$, it holds with probability at least $1 - (1 - p)/4$ that for all $r \in \hlint{0}{1}$ and all $u,v \in \AA_{3r/4, 5r/4}(0)$,
        \begin{align*}
          \frac{1}{2} \LFPP\left(u,v; \AA_{3r/4, 5r/4}(0) \right) \
          &\leq \ \locLFPP\left(u,v; \AA_{3r/4, 5r/4}(0) \right) \
          \leq \ 2 \LFPP\left(u,v; \AA_{3r/4, 5r/4}(0) \right).
        \end{align*}
        Using the convergence in probability of $\LFPP$ to $D_h$, we can decrease $\epsilon_0$ further, decrease $s$, and increase $S$ so that for each $\epsilon \in (0, \epsilon_0^{\zeta})$, it holds with probability at least $1 - (1 - p)/4$ that
        \begin{align*}
          \LFPP\left(\partial \AA_{7/8, 1}(0), \partial \AA_{3/4, 5/4}(0) \right) \
          &\geq \ 2s, \\
          \sup_{u,v \in \overline{\AA_{7/8, 1}(0)}} \LFPP\left(u,v; \AA_{3/4, 5/4}(0) \right) \
          &< \ \frac{S}{2}.
        \end{align*}
        Choose $\alpha_{*} \in (7/8, 1)$ and shrink $\epsilon_0$ according to Lemma \ref{lemma:LowerBoundForDistancesInANarrowAnnulus} with $1 - (1 - p)/4$ in place of $p$.
        Fix $\alpha \in \hrint{\alpha_{*}}{1}$, $r \in (\epsilon^{1 - \zeta},1)$, and $z \in \CC$.
        With probability at least $p$, if $u \in \partial B_{\alpha r}(z)$ and $v \in \partial B_r(z)$ satisfy $D_h(u,v) > D_h(u, \partial \AA_{r/2, 2r}(z))$, then
        \begin{align*}
          D_h\left(u,v \right) \
          &> \ D_h\left(\partial \AA_{3r/4, r}(z), \partial \AA_{r/2, 2r}(z) \right) \
          \geq \ s r^{\xi Q} e^{\xi h_r(z)},
        \end{align*}
        so
        \begin{align*}
          \locLFPP\left(u,v; \overline{\AA_{\alpha r, r}(z)} \right) \
          &\geq \ S \ratios r^{\xi Q} e^{\xi h_r(z)} \
          > \ \locLFPP\left(u,v; \AA_{r/2 + \frepsilon, 2r - \frepsilon}(z) \right).
        \end{align*}
        If $u \in \partial B_{\alpha r}(z)$ and $v \in \partial B_r(z)$ satisfy $\locLFPP(u,v) > \locLFPP(u, \partial \AA_{r/2 + \frepsilon, 2r - \frepsilon}(z))$, then
        \begin{align*}
          \locLFPP\left(u,v \right) \
          &> \ \locLFPP\left(\partial \AA_{7r/8, r}(z), \partial \AA_{3r/4, 5r/4}(z) \right) \
          \geq \ s \ratios r^{\xi Q} e^{\xi h_r(z)},
        \end{align*}
        so 
        \begin{align*}
          D_h\left(u,v; \overline{\AA_{\alpha r, r}(z)} \right) \
          &\geq \ S r^{\xi Q} e^{\xi h_r(z)} \
          > \ D_h\left(u,v; \AA_{r/2, 2r}(z) \right).
        \end{align*}
      \end{proof}

      \begin{lemma}
        \label{lemma:Condition3}
        Fix $p \in (0,1)$ and $\alpha \in (7/8, 1)$.
        There exists $A = A(p, \alpha) > 0$ and $\epsilon_0 = \epsilon_0(p) \in (0,1)$ such that for each $z \in \CC$, each $\epsilon \in (0,\epsilon_0)$, and each $r \in [\epsilon, 1]$, $E_{r,\epsilon}^{(3)}(z)$ occurs with probability at least $p$.
        \begin{proof}
          We'll first prove the claim for $D_h$, in which case axiom \ref{axiom:CoordinateChange} implies it will be enough to choose $A$ so that with probability at least $p$, $D_h(\text{around } \AA_{\alpha, 1}(0)) < A D_h(\partial B_{\alpha}(0), \partial B_1(0))$.
          This is true because $D_h(\partial B_{\alpha}(0), \partial B_1(0)) > 0$ and $D_h(\text{around } \AA_{\alpha, 1}(0)) < \infty$ almost surely.

          To prove the claim for $\locLFPP$, it is enough by \ref{loc:TranslationInvariance} from Lemma \ref{lemma:PropertiesOfLocalizedFieldAndLFPP} to consider the case $z =0 $.
          Convergence in probability of $\LFPP$ to $D_h$ and the fact that $D_h$ induces the Euclidean topology imply $C$ can be chosen large enough that for each $\epsilon \in \hlint{0}{1}$, it holds with probability at least $1 - (1-p)/2$ that
          \begin{align*}
            \LFPP\left(\text{around } \AA_{\alpha, 1}(0) \right) \
            &\leq \ C, \\
            \LFPP\left(\partial B_{\alpha}(0), \partial B_1(0) \right) \
            &\geq \ C^{-1}.
          \end{align*}
          By \eqref{eq:LFPPScaling} and scale invariance of the GFF, for each $\epsilon \in (0, 1)$ and each $r \in [\epsilon, 1]$, it holds with probability at least $1 - (1-p)/2$ that
          \begin{align*}
            \LFPP\left(\text{around } \AA_{\alpha r, r}(0) \right) \
            &\leq \ C \ratios r^{\xi Q} e^{\xi h_r(0)}, \\
            \LFPP\left(\partial B_{\alpha r}(0), \partial B_r(0) \right) \
            &\geq \ C^{-1} \ratios r^{\xi Q} e^{\xi h_r(0)}.
          \end{align*}
          Use \ref{loc:UniformComparison} from Lemma \ref{lemma:PropertiesOfLocalizedFieldAndLFPP} to choose $\epsilon_0 \in (0,1)$ such that for each $\epsilon \in (0, \epsilon_0)$, it holds with probability at least $1 - (1 - p)/2$ that
          \begin{align*}
            \frac{1}{2} D_h^{\epsilon}(u,v; B_1(0)) \
            &\leq \ \hat{D}_h^{\epsilon}(u,v; B_1(0)) \
            \leq \ 2 D_h^{\epsilon}(u,v; B_1(0)) \ \forall u,v \in B_1(0).
          \end{align*}
          The claim follows by taking any $A \geq 4 C^2$.
        \end{proof}
      \end{lemma}

      \begin{lemma}
        \label{lemma:Condition4}
        Fix $\alpha \in (7/8, 1)$, $\delta, p \in (0,1)$, and $0 < \zeta_{-} < \zeta < 1$.
        There exists $\epsilon_0 \in (0,1)$ such that for every $\epsilon \in (0, \epsilon_0)$, every $r \in (\epsilon^{1 - \zeta}, 1)$, and every $z \in \CC$, $E_{r,\epsilon}^{(4)}(z)$ occurs with probability at least $p$.

        \begin{proof}
          By axiom \ref{axiom:CoordinateChange} and \ref{loc:TranslationInvariance} from Lemma \ref{lemma:PropertiesOfLocalizedFieldAndLFPP}, it will suffice to chose $\epsilon_0$ so that the claim holds with $z = 0$.
          Using the convergence in probability of $\LFPP$ and $\locLFPP$ to $D_h$, we can find $R > 0$ and $\epsilon_0 \in (0, 1)$ such that for each $\epsilon \in (0, \epsilon_0)$, it holds with probability at least $1 - (1 - p)/3$ that
          \begin{align}
            D_h^{\epsilon}\left(u,v \right) \
            &= \ D_h^{\epsilon}\left(u,v; B_R(0) \right) \ \forall u,v \in B_1(0), \label{eq:Condition4InternalMetricLFPP} \\
            \hat{D}_h^{\epsilon}\left(u,v \right) \
            &= \ \hat{D}_h^{\epsilon}\left(u,v; B_R(0) \right) \ \forall u,v \in B_1(0). \label{eq:Condition4InternalMetricLocalizedLFPP}
          \end{align}
          By \ref{loc:UniformComparison} from Lemma \ref{lemma:PropertiesOfLocalizedFieldAndLFPP}, we can decrease $\epsilon_0$ so that for each $\epsilon \in (0, \epsilon_0)$, it holds with probability at least $1 - (1 - p)/3$ that
          \begin{align}
            \frac{1}{2} \LFPP(u,v; B_R(0)) \
            &\leq \ \locLFPP(u,v; B_R(0)) \
            \leq \ 2 \LFPP(u,v; B_R(0)) \ \forall u,v \in B_R(0). \label{eq:Condition4CompareLFPPAndLocalizedLFPP}
          \end{align}
          Using the convergence in probability of $\LFPP$ to $D_h$, we can shrink $\epsilon_0$ so that for each $\epsilon \in (0, \epsilon_0^{\zeta})$, it holds with probability at least $1 - (1 - p)/3$ that
          \begin{align}
            \LFPP\left(u,v \right) \
            &\leq \ \frac{\delta}{2} D_h\left(\partial B_{\alpha}(0), \partial B_1(0) \right) \ \forall u, v \in B_2(0), |u - v| \leq 4 \epsilon_0^{\zeta - \zeta_{-}}, \label{eq:Condition4LFPPError} \\
            D_h(u,v) \
            &\leq \ \frac{\delta}{2} \LFPP\left(\partial B_{\alpha}(0), \partial B_1(0) \right) \ \forall u,v \in B_2(0), |u - v| \leq 4 \epsilon_0^{\zeta - \zeta_{-}}. \label{eq:Condition4DhError}
          \end{align}

          Now if $\epsilon \in (0, \epsilon_0)$ and $r \in (\epsilon^{1 - \zeta}, 1)$, then $\epsilon/r \in (0, \epsilon_0^{\zeta})$, so with probability at least $p$, for all $u \in \overline{\AA_{\alpha r, r}(0)}$ and $v \in \overline{B_{4 \epsilon^{1 - \zeta_{-}}}(u)}$,
          \begin{align*}
            \locLFPP(u,v) \
            &\leq \ 2 \LFPP(u,v) & (\text{by \eqref{eq:Condition4InternalMetricLFPP}, \eqref{eq:Condition4InternalMetricLocalizedLFPP}, \eqref{eq:Condition4CompareLFPPAndLocalizedLFPP}}) \\
            &= \ 2 \ratios r^{\xi Q} e^{\xi h_r(0)} \LFPP[\epsilon/r][h(r \cdot) - h_r(0)]\left(u/r, v/r \right) & (\text{by \eqref{eq:LFPPScaling}}) \\
            &\leq \ \delta \ratios r^{\xi Q} e^{\xi h_r(0)} D_{h(r \cdot) - h_r(0)} \left(\partial B_{\alpha}(0), \partial B_1(0) \right) & (\text{by \eqref{eq:Condition4LFPPError}}) \\
            &= \ \delta \ratios D_h\left(\partial B_{\alpha r}(0), \partial B_r(0) \right) & (\text{by axioms \ref{axiom:WeylScaling} and \ref{axiom:CoordinateChange}}).
          \end{align*}
          Likewise,
          \begin{align*}
            D_h\left(u,v \right) \
            &= \ r^{\xi Q} e^{\xi h_r(0)} D_{h(r \cdot) - h_r(0)}\left(u/r, v/r\right) & (\text{by axioms \ref{axiom:WeylScaling} and \ref{axiom:CoordinateChange}}) \\
            &\leq \ \frac{\delta}{2} r^{\xi Q} e^{\xi h_r(0)} \LFPP[\epsilon/r][h(r \cdot) - h_r(0)]\left(\partial B_{\alpha}(0), \partial B_1(0) \right) & (\text{by \eqref{eq:Condition4DhError}}) \\
            &= \ \frac{\delta}{2} \invratios \LFPP\left(\partial B_{\alpha r}(0), \partial B_r(0) \right) & (\text{by \eqref{eq:LFPPScaling}}) \\
            &\leq \ \delta \invratios \locLFPP\left(\partial B_{\alpha r}(0), \partial B_r(0) \right) & (\text{by \eqref{eq:Condition4InternalMetricLFPP}, \eqref{eq:Condition4InternalMetricLocalizedLFPP}, \eqref{eq:Condition4CompareLFPPAndLocalizedLFPP}}). 
          \end{align*}
        \end{proof}
      \end{lemma}

      \begin{proof}[Proof of Lemma \ref{lemma:ChoosingParameters}]
        Choose $p$ and $c$ according to Lemma \ref{lemma:IndependenceAcrossConcentricAnnuli} using the parameters $a = 24 \log 8$, $s_1 = 1/8$, and $s_2 = 1/2$.
        Choose $A$, $\alpha$, and $\epsilon_0$ according to Lemmas \ref{lemma:Condition1}, \ref{lemma:Condition2}, \ref{lemma:Condition3}, and \ref{lemma:Condition4}, with $1 - (1 - p)/4$ in place of $p$, so $P[E_{r,\epsilon}(z)] \geq p$ for every $r \in (\epsilon^{1 - \zeta}, \epsilon^{1 - \zeta_{+}}) \cap \{8^j\}_{j \in \ZZ}$, every $\epsilon \in (0, \epsilon_0)$, and every $z \in \CC$.
        In particular, $A$ depends on $p$ and $\alpha$, and $\alpha$ depends only on $p$, so $A$ depends only on $p$.
        For each $\epsilon \in (0,1)$, fix $\cR_{\epsilon} \subset (\epsilon^{1 - \zeta}, \epsilon^{1 - \zeta_{+}}) \cap \{8^j\}_{j \in \ZZ}$ with $\# \cR_{\epsilon} > \#((\epsilon^{1 - \zeta}, \epsilon^{1 - \zeta_{+}}) \cap \{8^j\}_{j \in \ZZ})$.

        Fix $\epsilon \in (0, \epsilon_0)$ and $z \in \CC$.
        In the notation of Lemma \ref{lemma:IndependenceAcrossConcentricAnnuli}, take the sequence $(r_k/4)_{k=1}^{\infty}$ to be the elements of $\cR_{\epsilon} \cup ((0, \epsilon^{1 - \zeta}) \cap \{8^j\}_{j \in \ZZ})$, and the events $E_{r_k}$ to be $E_{r_k/4, \epsilon}(z)$ when $r_k/4 \in \cR_{\epsilon}$ and $E_{r_k}$ the trivial event otherwise, so
        \begin{align*}
          P\left\{E_{r,\epsilon}(z) \text{ occurs for at least one } r \in \cR_{\epsilon} \right\} \
          &\geq \ 1 - c e^{-24 \log(8) \# \cR_{\epsilon}}.
        \end{align*}
        The condition $1 - \zeta_{+} < \g (1 - \zeta)$ implies
        \begin{align*}
          \# \cR_{\epsilon} \
          &> \ \frac{1}{3} \# \left(\left(\epsilon^{1 - \zeta}, \epsilon^{1 - \zeta_{+}} \right) \cap \{8^j\}_{j \in \ZZ} \right) \
          \geq \ \frac{1 - \zeta}{6 \log 8} \log \epsilon^{-1} - \frac{1}{3},
        \end{align*}
        so
        \begin{align*}
          P\left\{E_{r,\epsilon}(z) \text{ occurs for at least one } r \in \cR_{\epsilon} \right\} \
          &\geq \ 1 - c \epsilon^{4(1 - \zeta)},
        \end{align*}
        where we have absorbed the $e^{8 \log(8)}$ into the constant $c$.
      \end{proof}
      
    \subsection{Proof of Proposition \ref{prop:MainIterationArgument}}
      Before turning to the proof of Proposition \ref{prop:MainIterationArgument}, we remind the reader that $\hyperlink{V}{V(R)}$ and $\hyperlink{Vhat}{\hat{V}_{\epsilon}(R)}$ denote the sets of pairs $(z,w) \in \CC^2$ such that there is a $D_h$- (resp. $\hat{D}_h^{\epsilon}$-) geodesic from $z$ to $w$ contained within $B_R(0)$.

      \newcommand{\lattice}{\frac{\epsilon^{1 - \zeta}}{4} \ZZ^2} 

      Choose $A$, $\alpha$, and $\epsilon_0$ according to Lemma \ref{lemma:ChoosingParameters}.
      Let $H_{\epsilon}$ denote the event that for every $w \in B_{R+1}(0) \cap \lattice$, there exists $r \in \cR_{\epsilon}$ such that $E_{r,\epsilon}(w)$ occurs.
      Then by a union bound, $P[H_{\epsilon}] = 1 - O_{\epsilon}(\epsilon^{2(1 - \zeta)})$, where the big-$O$ constant depends only on $R$ and $\zeta$.
      It follows that if $F_{\epsilon}$ is the event that \eqref{eq:InitialLipschitzConstant} holds, then $P[F_{\epsilon} \cap H_{\epsilon}] = 1 - O_{\epsilon}(\epsilon^{\beta \w 2(1 - \zeta)})$.
      We will now show that on $F_{\epsilon} \cap H_{\epsilon}$, \eqref{eq:NewLipschitzConstant} holds (\eqref{eq:NewLipschitzConstantReversed} follows by swapping the roles of $D_h$ and $\locLFPP$).
      We may assume $C_0(\epsilon) \geq \sup_{r \in \cR_{\epsilon}} \ratios$.

      Fix $(z,w) \in \hyperlink{V}{V(R)}$ and let $P \colon [0, D_h(z,w)] \to B_R(0)$ be a $D_h$-geodesic from $z$ to $w$ parameterized by $D_h$-length and contained in $B_R(0)$.
      Using the fact that $\{B_{\epsilon^{1 - \zeta}/2}(x) : x \in L_{\epsilon}(\zeta) \cap B_{R+1}(0)\}$ covers $B_R(0)$, we can break $P$ into segments as follows.
      Let $t_0 \coloneqq 0$ and choose some $x_0 \in B_{R+1}(0) \cap \lattice$ and $r_0 \in \cR_{\epsilon}$ such that $z \in B_{r_0/2}(x_0)$ and $E_{r_0, \epsilon}(x_0)$ occurs.
      Then inductively define $t_j$ to be the first time after $t_{j-1}$ that $P$ leaves $B_{r_{j-1}}(x_{j-1})$, and choose some $x_j \in B_{R+1}(0) \cap \lattice$ and some $r_j \in \cR_{\epsilon}$ such that $P(t_j) \in B_{r_j/2}(x_j)$ and $E_{r_j, \epsilon}(x_j)$ occurs.
      If no such time exists, instead let $t_j \coloneqq D_h(z,w)$ and leave $x_j$ undefined.
      If such a time $t_j$ does exist, let $s_j$ be the last time before $t_j$ that $P$ exits $B_{\alpha r_{j-1}}(x_{j-1})$.
      Define
      \begin{align*}
        \underline{J} \
        &\coloneqq \ \max\left\{j \geq 1 : |z - P(t_{j-1})| \leq \frac{3 \epsilon^{1 - \zeta_{+}}}{2} \right\}, \\
        \overline{J} \
        &\coloneqq \ \min\left\{j \geq 0: |w - P(t_{j+1})| \leq \frac{3 \epsilon^{1 - \zeta_{+}}}{2} \right\}.
      \end{align*}
      Note that if $t_0 < t_j \leq D_h(z,w)$, then $P(t_j) \in \overline{B_{r_{j-1}}(x_{j-1})}$ and $P(t_{j-1}) \in B_{r_{j-1}/2}(x_{j-1})$, so
      \begin{align*}
        \left|P(t_j) - P(t_{j-1}) \right| \
        &\leq \ \frac{3 \epsilon^{1 - \zeta_{+}}}{2}.
      \end{align*}
      It follows that
      \begin{align}
        \left|P(t_{\underline{J}}) - z \right| \
        &\leq \ \left|P(t_{\underline{J}}) - P(t_{\underline{J} - 1}) \right| + \left|P(t_{\underline{J} - 1}) - z \right| \
        \leq \ 3 \epsilon^{1 - \zeta_{+}}, \label{eq:StartOfGeodesicBound} \\
        \left|P(t_{\overline{J}}) - w \right| \
        &\leq \ \left|P(t_{\overline{J}}) - P(t_{\overline{J} + 1}) \right| + \left|P(t_{\overline{J} + 1}) - w \right| \
        \leq \ 3 \epsilon^{1 - \zeta_{+}}, \label{eq:EndOfGeodesicBound} 
      \end{align}

      \begin{figure}[h]
        \centering
        \includegraphics[scale=0.7]{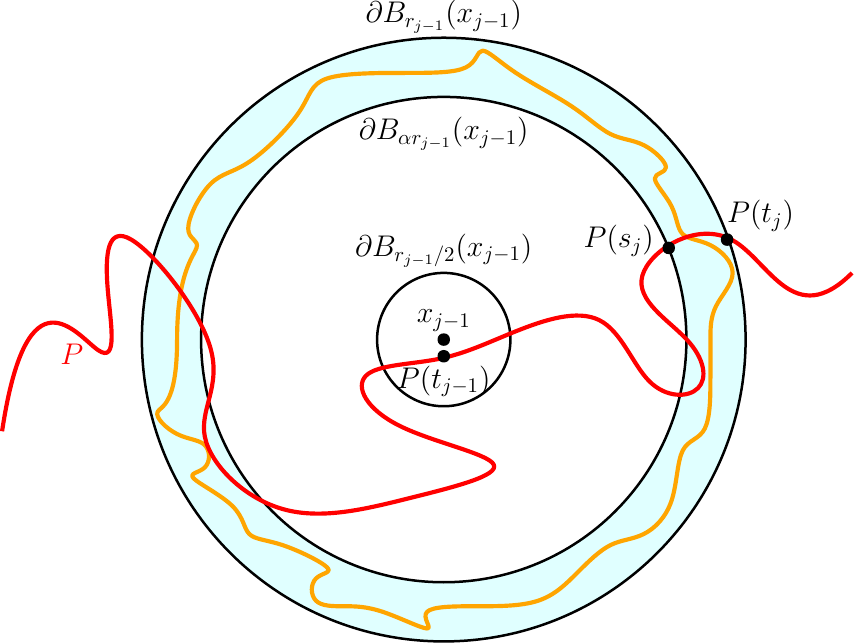}
        \caption{Illustration of the proof of Proposition \ref{prop:MainIterationArgument}.
        Pictured is one of the annuli $\AA_{\alpha r_{j-1}, r_{j-1}}(x_{j-1})$ such that $E_{r_{j-1}, \epsilon}(x_{j-1})$ occurs.
        The path in red is the $D_h$-geodesic $P$.
        Time $t_{j}$ is the first time after time $t_{j-1}$ that $P$ exits $B_{r_{j-1}}(x_{j-1})$, and $s_j$ is the last time before $t_j$ that $P$ leaves $B_{\alpha r_{j-1}}(x_{j-1})$.
        The segment $P|_{[s_j, t_j]}$ is a $D_h$-geodesic in $\AA_{\alpha r_{j-1}, r_{j-1}}(x_{j-1})$, so condition 1 in Definition \ref{defn:AnnulusEvents} says $\locLFPP(P(s_j), P(t_j)) \leq \frac{r_{j-1} \fa_{\epsilon}^{-1}}{r_{j-1}^{\xi Q} \fa_{\epsilon/r_{j-1}}^{-1}} (1 + \delta) (t_j - s_j)$.
        The orange path is from condition 3 in Definition \ref{defn:AnnulusEvents}, and has $D_h$-length at most $A D_h(\partial B_{\alpha r_{j-1}}(x_{j-1}), \partial B_{r_{j-1}}(x_{j-1})) \leq A (t_j - s_j)$.
        Since $P$ crosses the orange path before time $t_{j-1}$ and after time $s_j$, it follows that $s_j - t_{j-1} \leq A(t_j - s_j)$.
        We use the latter to show that a positive proportion of $P$ is comprised of the segments $P|_{[s_j, t_j]}$.
        }
      \end{figure}

      If $j \in [\underline{J} + 1, \overline{J}] \cap \ZZ$, then $P|_{[s_j, t_j]}$ is a $D_h$-geodesic contained in $\overline{\AA_{\alpha r_{j-1}, r_{j-1}}(x_{j-1})}$, so using condition \ref{conditions:GoodLipschitzConstant} in Definition \ref{defn:AnnulusEvents},
      \begin{align}
        \locLFPP(P(s_j), P(t_j)) \
        &\leq \ \ratios (1 + \delta) \left(t_j - s_j \right).
        \label{eq:MainIterationGoodSegments}
      \end{align}
      If $j \in [\underline{J}, \overline{J} - 1] \cap \ZZ$, then we can bound $\locLFPP(P(t_j), P(s_{j+1}))$ using \eqref{eq:InitialLipschitzConstant} and condition \ref{conditions:Error} in Definition \ref{defn:AnnulusEvents}:
      \begin{align*}
        \locLFPP(P(t_j), P(s_{j+1})) \
        &\leq \ \locLFPP\left(B_{4 \epsilon^{1 - \zeta_{-}}}(P(t_j)), B_{4 \epsilon^{1 - \zeta_{-}}}(P(s_{j+1})) \right) \\
        &\qquad \qquad + \sup_{u \in \overline{\AA_{\alpha r_{j-1}, r_{j-1}}(x_{j-1})}} \sup_{v \in \overline{B_{4 \epsilon^{1 - \zeta_{-}}}(u)}} \locLFPP(u,v) \\
        &\qquad \qquad + \sup_{u \in \overline{\AA_{\alpha r_{j}, r_{j}}(x_{j})}} \sup_{v \in \overline{B_{4 \epsilon^{1 - \zeta_{-}}}(u)}} \locLFPP(u,v) \\
        &\leq \ C_0(\epsilon) (s_{j+1} - t_j) \\
        &\qquad \qquad + \delta \ratios D_h\left(\partial B_{\alpha r_{j-1}}(x_{j-1}), \partial B_{r_{j-1}}(x_{j-1}) \right) \\
        &\qquad \qquad + \delta \ratios D_h\left(\partial B_{\alpha r_{j}}(x_{j}), \partial B_{r_{j}}(x_{j}) \right).
      \end{align*}
      For each $j \in [\underline{J} - 1, \overline{J} -1] \cap \ZZ$, $P$ crosses $\AA_{\alpha r_j, r_j}(x_j)$ at least once between time $t_{j}$ and $t_{j+1}$, so
      \begin{align*}
        D_h(\partial B_{\alpha r_j}(x_j), \partial B_{r_j}(x_j)) \
        &\leq \ t_{j+1} - t_j.
      \end{align*}
      Thus, for all $j \in [\underline{J}, \overline{J} - 1] \cap \ZZ$,
      \begin{align}
        \locLFPP\left(P(t_j), P(s_{j+1}) \right) \
        &\leq \ C_0(\epsilon)\left(s_{j+1} - t_j \right) + \delta \ratios (t_{j} - t_{j-1}) + \delta \ratios (t_{j+1} - t_j) 
        \label{eq:MainIterationBadSegments}
      \end{align}
      Summing over $j \in [\underline{J}, \overline{J} - 1] \cap \ZZ$, we obtain
      \begin{align}
        \locLFPP(P(t_{\underline{J}}), P(t_{\overline{J}})) \
        &\leq \ \sum_{j=\underline{J}}^{\overline{J} - 1} \locLFPP\left(P(t_j), P(s_{j+1}) \right) + \locLFPP\left(P(s_{j+1}), P(t_{j+1}) \right) \nonumber \\
        &\leq \ \sup_{r \in \cR_{\epsilon}} \ratios (1 + \delta) \sum_{j=\underline{J} + 1}^{\overline{J}} t_j - s_j & (\text{by \eqref{eq:MainIterationGoodSegments}}) \nonumber \\
        &\qquad + C_0(\epsilon) \sum_{j=\underline{J}+1}^{\overline{J}} s_{j} - t_{j-1} + 2 \delta \sup_{r \in \cR_{\epsilon}} \ratios \sum_{j=\underline{J}}^{\overline{J}} t_{j+1} - t_j & (\text{by \eqref{eq:MainIterationBadSegments}}) \nonumber \\
        &= \ \sup_{r \in \cR_{\epsilon}} \ratios (1 + \delta) \sum_{j=\underline{J}+1}^{\overline{J}} t_j - t_{j-1} \nonumber \\
        &\qquad + \left[C_0(\epsilon) - \sup_{r \in \cR_{\epsilon}} \ratios (1 + \delta) \right] \sum_{j=\underline{J}+1}^{\overline{J}} s_j - t_{j-1} \nonumber \\
        &\qquad + 2 \delta \sup_{r \in \cR_{\epsilon}} \ratios \sum_{j=\underline{J}}^{\overline{J}} t_{j+1} - t_j. \label{eq:LocalizedLFPPTriangleInequality}
      \end{align}

      If $j \in [\underline{J} + 1, \overline{J}] \cap \ZZ$, then $|P(t_{j-1}) - z| > 3 r_{j-1}/2$ and $|P(t_{j-1}) - x_{j-1}| < r_{j-1}/2$, so $|z - x_{j-1}| > r_{j-1}$.
      Likewise, $|w - x_{j-1}| > r_{j-1}$.
      So $P$ must cross $\AA_{\alpha r_{j-1}, r_{j-1}}(x_{j-1})$ at least once before time $t_{j-1}$ and at least once after time $s_j$.
      Condition \ref{conditions:Proportions} in Definition \ref{defn:AnnulusEvents} then implies 
      \begin{align*}
        s_j - t_{j-1} \
        &\leq \ A D_h\left(\partial B_{\alpha r_{j-1}}(x_{j-1}), \partial B_{r_{j-1}}(x_{j-1}) \right) \
        \leq \ A \left(t_j - s_j \right).
      \end{align*}
      Adding $A(s_j - t_{j-1})$ to both sides, then dividing through by $A+1$ yields
      \begin{align*}
        s_j - t_{j-1} \
        &\leq \ \frac{A}{A+1} \left(t_j - t_{j-1} \right).
      \end{align*}
      Combining this estimate with \eqref{eq:LocalizedLFPPTriangleInequality}, \eqref{eq:StartOfGeodesicBound}, and \eqref{eq:EndOfGeodesicBound} yields
      \begin{align*}
        \locLFPP\left(B_{4 \epsilon^{1 - \zeta_{+}}}(z), B_{4 \epsilon^{1 - \zeta_{+}}}(w) \right) \
        &\leq \ \sup_{r \in \cR_{\epsilon}} \ratios (1 + \delta) \left(t_{\overline{J}} - t_{\underline{J}} \right) \\
        &\qquad \qquad + \frac{A}{A+1} \left[C_0(\epsilon) - \sup_{r \in \cR_{\epsilon}} \ratios (1 + \delta) \right] \left(t_{\overline{J}} - t_{\underline{J}} \right) \\
        &\qquad \qquad + 2 \delta \sup_{r \in \cR_{\epsilon}} \ratios \left(t_{\overline{J} + 1} - t_{\underline{J}} \right) \\
        &\leq \ \left[\frac{A}{A+1} C_0(\epsilon) + \left(\frac{1 + \delta}{A+1} + 2 \delta \right) \sup_{r \in \cR_{\epsilon}} \ratios \right] D_h(z,w).
      \end{align*}

    \subsection{Choosing Good Scaling Ratios}
      Our next goal is to iterate Proposition \ref{prop:MainIterationArgument} to prove the following.

      \begin{prop}
        \label{prop:LipschitzConstant1}
        Fix $\delta \in (0,1)$.
        There exists $\beta = \beta(\delta) > 0$ and $\zeta = \zeta(\delta) \in (0,1)$ such that with probability $1 - O_{\epsilon}(\epsilon^{\beta})$ as $\epsilon \to 0$,
        \begin{align}
          \locLFPP\left(B_{4 \epsilon^{1 - \zeta}}(u), B_{4 \epsilon^{1 - \zeta}}(v) \right) \
          &\leq \ \left(1 + \delta \right) D_h\left(u,v \right) \ \forall (u,v) \in \hyperlink{V}{V(R)},
          \label{eq:LipschitzConstant1LocalDh} \\
          D_h\left(B_{4 \epsilon^{1 - \zeta}}(u), B_{4 \epsilon^{1 - \zeta}}(v) \right) \
          &\leq \ \left(1 + \delta \right) \locLFPP\left(u,v \right) \ \forall (u,v) \in \hyperlink{Vhat}{\hat{V}_{\epsilon}(R)}.
          \label{eq:LipschitzConstant1DhLocal}
        \end{align}
      \end{prop}

      We remark that when $\delta$ is very small, $\beta$ and $1 - \zeta$ will both be very close to $0$.
      The main obstacle in deducing Proposition \ref{prop:LipschitzConstant1} from Proposition \ref{prop:MainIterationArgument} is that, a priori, the $\sup_{r \in \cR_{\epsilon}} \ratios$ and $\sup_{r \in \cR_{\epsilon}} \invratios$ terms could be quite far from $1$.
      So we will show that the proportion of ``bad'' radii where these terms are far from $1$ is no larger than $\frac{2}{3}$.

      \begin{lemma}
        \label{lemma:GoodRatios}
        Let $A$ be as in Proposition \ref{prop:MainIterationArgument}.
        Choose a sequence $0 < \zeta_{+}(0) = \zeta_{-}(1) < \zeta(1) < \zeta_{+}(1) = \zeta_{-}(2) < \zeta(2) < \zeta_{+}(2) = \zeta_{-}(3) < \cdots < 1$ with $1 - \zeta_{+}(n) < \g(1 - \zeta(n))$ for all $n$.
        Let $C_0$ be as in Proposition \ref{prop:UpToConstants} with $\zeta = \zeta_{+}(0)$.
        Let
        \begin{align*}
          \cS_{\epsilon}^{(n)} \
          &\coloneqq \ \left\{\cR \subset (\epsilon^{1 - \zeta(n)}, \epsilon^{1 - \zeta_{+}(n)}) \cap \{8^j\}_{j \in \ZZ} : \# \cR \geq \frac{1}{3} \# ((\epsilon^{1 - \zeta(n)}, \epsilon^{1 - \zeta_{+}(n)}) \cap \{8^j\}_{j \in \ZZ})\right\},
        \end{align*}
        and define
        \begin{align*}
          X_{\epsilon}^{(n)} \
          &\coloneqq \ \min_{\cR \in \cS_{\epsilon}^{(n)}} \sup_{r \in \cR} \ratios, \\
          Y_{\epsilon}^{(n)} \
          &\coloneqq \ \min_{\cR \in \cS_{\epsilon}^{(n)}} \sup_{r \in \cR} \invratios.
        \end{align*}
        Fix $\delta \in (0,1/2)$ and choose $N = N(\delta)$ large enough that $(\frac{A}{A+1})^{N} C_0 < \delta^2/2$.
        There exists $\epsilon_N \in (0,1)$ such that for all $\epsilon \in (0, \epsilon_N)$, 
        \begin{align*}
          \# \left\{1 \leq n \leq 3N : X_{\epsilon}^{(n)} > 1 - \delta \right\} \
          &> \ 2N, \\
          \# \left\{1 \leq n \leq 3N : Y_{\epsilon}^{(n)} > 1 - \delta \right\} \
          &> \ 2N.
        \end{align*}

        \begin{proof}
          The idea of the proof is that if there is a sequence of $\epsilon$'s tending to $0$ along which at least $N$ of the $X_{\epsilon}^{(n)}$ are $\leq 1 - \delta$, then there is a subset $S \subset \{1, 2, \ldots, 3N\}$ with $\# S \geq N$ and a subsequence of $\epsilon$'s along which $X_{\epsilon}^{(n)} \leq 1 - \delta$ for all $n \in S$.
          Applying Proposition \ref{prop:MainIterationArgument} $N$ times with the exponents $\{\zeta_{+}(n) : n \in S\}$ will then show that along the given subsequence, $\locLFPP$ and $D_h$ satisfy a Lipschitz condition with Lipschitz constant strictly less than $1$, contradicting convergence in probability.
          The proof is similar for $X_{\epsilon}^{(n)}$ and $Y_{\epsilon}^{(n)}$, so we will do the $Y_{\epsilon}^{(n)}$ case.

          Assume there is a sequence $\cE = \cE(N)$ of $\epsilon$'s tending to $0$ along which $\#\{1 \leq n \leq N : Y_{\epsilon}^{(n)} \leq 1 - \delta\} \geq N$.
          For each set $S \subset \{1,2,\ldots, 3 N\}$ with $\# S \geq N$, let
          \begin{align*}
            \cE_S \
            &\coloneqq \ \left\{\epsilon \in \cE : S \subset \{1 \leq n \leq 3 N : Y_{\epsilon}^{(n)} \leq 1 - \delta \} \right\}.
          \end{align*}
          For each $\epsilon \in \cE$, we have $\epsilon \in \cE_{S}$ for $S = \{1 \leq n \leq 3 N : Y_{\epsilon}^{(n)} \leq 1 - \delta\}$.
          So $\cE$ is the union of the $\cE_S$ over all $S$ with $\# S \geq N$.
          It follows that there exists $S$ with $\# S \geq N$ such that $\# \cE_S = \infty$.
          Enumerate the first $N$ elements of $S$ as $j_1 < j_2 < \cdots < j_N$, and put $j_0 \coloneqq 0$.

          For each $\epsilon \in \cE_S$, we have $Y_{\epsilon}^{(j_k)} \leq 1 - \delta$ for all $1 \leq k \leq N$.
          Now iteratively apply Proposition \ref{prop:MainIterationArgument} with $(\zeta_{-}, \zeta, \zeta_{+}) = (\zeta_{+}(j_{k-1}), \zeta(j_k), \zeta_{+}(j_k))$ with $k \in \{1,2,3,\ldots, N\}$, with the radii sets $\cR_{\epsilon}^{(k)}$ equal to the minimizers of $Y_{\epsilon}^{(j_k)}$, with the initial Lipschitz constant $C_0(\epsilon) = C_0$, with 
          \begin{align*}
            \delta' \
            &\coloneqq \ \frac{\delta}{1 + 2(A+1)}
          \end{align*}
          in place of $\delta$, and with $R$ to be chosen later.
          It follows that for each $k \in \{1,2,\ldots, N\}$, there exists some $\beta_k > 0$ such that with probability at least $1 - O_{\epsilon}(\epsilon^{\beta_k})$ as $\epsilon \to 0$,
          \begin{align}
            D_h\left(B_{4 \epsilon^{1 - \zeta_{+}(j_k)}}(u), B_{4 \epsilon^{1 - \zeta_{+}(j_k)}}(v) \right) \
            &\leq \ C_k(\epsilon) \locLFPP\left(u,v \right) \ \forall (u,v) \in \hyperlink{Vhat}{\hat{V}_{\epsilon}(R)},
            \label{eq:GoodScalingRatiosLipschitzCondition}
          \end{align}
          where $C_k(\epsilon)$ is inductively defined by $C_0(\epsilon) = C_0$ and
          \begin{align*}
            C_{k+1}(\epsilon) \
            &\coloneqq \ \frac{A}{A+1} \left[Y_{\epsilon}^{(j_{k+1})} \vee C_k(\epsilon) \right] + \left[\frac{1 + \delta'}{A + 1} + 2 \delta' \right] Y_{\epsilon}^{(j_{k+1})}.
          \end{align*}

          We claim there exists $n \in \{0,1,2, \ldots, N\}$ such that $C_n(\epsilon) \leq 1 - \delta^2/2$ for all $\epsilon \in \cE_S$.
          Indeed, first note that if $C_k(\epsilon) \leq 1 - \delta^2/2$ for some $k \in \{0,1,2,\ldots, N-1\}$ and some $\epsilon \in \cE_S$, then
          \begin{align*}
            C_{k+1}(\epsilon) \
            &\leq \ \frac{A}{A+1}\left(1 - \frac{\delta^2}{2} \right) + \left[\frac{1 + \delta'}{A+1} + 2 \delta' \right] (1 - \delta) & (Y_{\epsilon}^{(j_{k+1})} \leq 1 - \delta) \\
            &= \ \frac{A}{A+1} \left(1 - \frac{\delta^2}{2} \right) + \frac{1 - \delta}{A+1} + \delta' (1 - \delta) \frac{1 + 2 (A + 1)}{A+1} \\
            &= \ \frac{A}{A+1} \left(1 - \frac{\delta^2}{2} \right) + \frac{1 - \delta}{A+1} + \frac{\delta(1 - \delta)}{A+1} & \left(\delta' = \frac{\delta}{1 + 2(A+1)} \right) \\
            &\leq \ 1 - \frac{\delta^2}{2}.
          \end{align*}
          So it will be enough to show that for every $\epsilon \in \cE_S$, there exists $n \in \{0,1,2, \ldots, N\}$ with $C_n(\epsilon) \leq 1 - \delta^2/2$.
          If $C_n(\epsilon) \leq 1 - \delta$ for some $n$, then this $n$ will suffice.
          Assuming $C_n(\epsilon) > 1 - \delta$ for all $n \in \{0,1,2, \ldots N\}$, then $Y_{\epsilon}^{(j_{n+1})} \vee C_n(\epsilon) = C_n(\epsilon)$ for all $n \in \{0,1,2, \ldots, N\}$ and hence by induction
          \begin{align*}
            C_N(\epsilon) \
            &= \ \left(\frac{A}{A+1} \right)^N C_0 + \sum_{k=0}^{N-1} \left(\frac{A}{A+1} \right)^k \left[\frac{1 + \delta'}{A+1} + 2 \delta' \right] Y_{\epsilon}^{(j_{N-k})} \\
            &\leq \ \left(\frac{A}{A+1} \right)^N C_0 + (A+1) \left[\frac{1 + \delta'}{A+1} + 2 \delta' \right] (1 - \delta) & (Y_{\epsilon}^{j_k} \leq 1 - \delta) \\
            &= \ \left(\frac{A}{A+1} \right)^N C_0 + (1 - \delta) + \delta' \left[1 + 2 (A + 1)\right](1 - \delta) \\
            &\leq \ \frac{\delta^2}{2} + (1 - \delta) + \delta(1 - \delta) & \left(\delta' = \frac{\delta}{1 + 2(A+1)} \right) \\
            &= \ 1 - \frac{\delta^2}{2}.
          \end{align*}

          It now follows from \eqref{eq:GoodScalingRatiosLipschitzCondition} with $k = n$ (using the fact that $C_n(\epsilon) \leq 1 - \delta^2/2$) that on $\{(0, \i) \in \hyperlink{Vhat}{\hat{V}_{\epsilon}(R)}\}$, except on an event of probability $O_{\epsilon}(\epsilon^{\beta_n})$ as $\epsilon \to 0$ through $\cE_S$,
          \begin{align*}
            \frac{\delta^2}{2} \locLFPP(0, \i)
            &\leq \left[\locLFPP\left(B_{4 \epsilon^{1 - \zeta_{+}(j_n)}}(0), B_{4 \epsilon^{1 - \zeta_{+}(j_n)}}(\i) \right) - D_h\left(B_{4 \epsilon^{1 - \zeta_{+}(j_n)}}(0), B_{4 \epsilon^{1 - \zeta_{+}(j_n)}}(\i) \right)\right] \\
            &\quad + \left[\locLFPP(0,\i) - \locLFPP\left(B_{4 \epsilon^{1 - \zeta_{+}(j_n)}}(0), B_{4 \epsilon^{1 - \zeta_{+}(j_n)}}(\i) \right)\right].
          \end{align*}
          The first term on the right-hand side converges to $0$ in probability, while the second term on the right-hand side is at most
          \begin{align*}
            \sup_{u \in \overline{B_{4 \epsilon^{1 - \zeta_{+}(n)}}(0)}} \locLFPP\left(0, u \right) + \sup_{v \in \overline{B_{4 \epsilon^{1 - \zeta_{+}(n)}}(\i)}} \locLFPP\left(v, \i \right),
          \end{align*}
          which also converges to $0$ in probability.
          But the left-hand side converges in probability to the positive random variable $\frac{\delta^2}{2} D_h(0, \i)$.
          We obtain a contradiction by choosing $R > 0$ and $\epsilon_0, p \in (0,1)$ such that $P\{(0, \i) \in \hyperlink{Vhat}{\hat{V}_{\epsilon}(R)}\} \geq p$ for all $\epsilon \in (0, \epsilon_0)$, which can be arranged using the convergence in probability of $\locLFPP$ to $D_h$.
        \end{proof}
      \end{lemma}

      \begin{proof}[Proof of Proposition \ref{prop:LipschitzConstant1}]
        Choose a sequence $0 < \zeta_{+}(0) = \zeta_{-}(1) < \zeta(1) < \zeta_{+}(1) = \zeta_{-}(2) < \zeta(2) < \zeta_{+}(2) = \zeta_{-}(3) < \cdots < 1$ with $1 - \zeta_{+}(n) < \g(1 - \zeta(n))$ for all $n$.
        Let $A$ be as in Proposition \ref{prop:MainIterationArgument}, and let $C_0$ be as in Proposition \ref{prop:UpToConstants} with $\zeta = \zeta_{+}(0)$.
        Choose $N$ large enough that $(\frac{A}{A+1})^N C_0 < \delta^2$, and let $\epsilon_N$ be as in Lemma \ref{lemma:GoodRatios} with $\frac{\delta}{1 + \delta}$ in place of $\delta$, so for every $\epsilon \in (0, \epsilon_N)$, the sets
        \begin{align*}
          U_{\epsilon} \
          &\coloneqq \ \left\{1 \leq n \leq 3 N : X_{\epsilon}^{(n)} > \frac{1}{1+\delta} \right\}, \\
          L_{\epsilon} \
          &\coloneqq \ \left\{1 \leq n \leq 3 N : Y_{\epsilon}^{(n)} > \frac{1}{1 + \delta} \right\}
        \end{align*}
        have size at least $2N$, hence $\# (U_{\epsilon} \cap L_{\epsilon}) \geq N$.
        For each set $S \subset \{1,2,\ldots, 3 N\}$ with $\# S \geq N$, let
        \begin{align*}
          \cE_S \
          &\coloneqq \ \left\{\epsilon \in (0, \epsilon_N) : S \subset U_{\epsilon} \cap L_{\epsilon} \right\},
        \end{align*}
        and note that because $\epsilon \in \cE_{U_{\epsilon} \cap L_{\epsilon}}$ for all $\epsilon \in (0, \epsilon_N)$, the interval $(0, \epsilon_N)$ is the union of the $\cE_S$ over all subsets $S \subset \{1,2, \ldots, 3N\}$ with $\# S \geq N$. 

        Fix $S \subset \{1,2,\ldots, 3 N\}$ with $\# S \geq N$ such that $\cE_S$ has $0$ as a limit point.
        For each $n \in S$ and each $\epsilon \in \cE_S$, the set 
        \begin{align}
          \cR \
          &\coloneqq \ \left\{r \in (\epsilon^{1 - \zeta(n)}, \epsilon^{1 - \zeta_{+}(n)}) \cap \{8^j\}_{j \in \ZZ} : \frac{1}{1 + \delta} \leq \ratios \leq 1 + \delta \right\}
          \label{eq:GoodScales}
        \end{align}
        has size $\geq \frac{1}{3} \# ((\epsilon^{1 - \zeta(n)}, \epsilon^{1 - \zeta_{+}(n)}) \cap \{8^j\}_{j \in \ZZ})$.
        Indeed, if this were not true, then one of the sets
        \begin{align*}
          \cL \
          &\coloneqq \ \left\{r \in (\epsilon^{1 - \zeta(n)}, \epsilon^{1 - \zeta_{+}(n)}) \cap \{8^j\}_{j \in \ZZ} : \ratios < \frac{1}{1+\delta} \right\}, \\
          \cU \
          &\coloneqq \ \left\{r \in (\epsilon^{1 - \zeta(n)}, \epsilon^{1 - \zeta_{+}(n)}) \cap \{8^j\}_{j \in \ZZ} : \invratios < \frac{1}{1 + \delta} \right\}
        \end{align*}
        must have size $\geq \frac{1}{3} \# ((\epsilon^{1 - \zeta(n)}, \epsilon^{1 - \zeta_{+}(n)}) \cap \{8^j\}_{j \in \ZZ})$.
        But if, say, $\# \cL \geq \frac{1}{3} \#((\epsilon^{1 - \zeta(n)}, \epsilon^{1 - \zeta_{+}(n)}) \cap \{8^j\}_{j \in \ZZ})$, then
        \begin{align*}
          \frac{1}{1+\delta} \
          &< \ X_{\epsilon}^{(n)} \
          \leq \ \sup_{r \in \cL} \ratios \
          \leq \ \frac{1}{1+\delta},
        \end{align*}
        a contradiction.

        Now enumerate the first $N$ elements of $S$ as $j_1 < j_2 < \cdots < j_N$.
        Iteratively apply Proposition \ref{prop:MainIterationArgument} using the initial Lipschitz constant $C_0$, with $(\zeta_{-}, \zeta, \zeta_{+}) = (\zeta_{-}(j_k), \zeta(j_k), \zeta_{+}(j_k))$ for $k \in \{1,2,\ldots, N\}$, with $\cR_{\epsilon}^{(k)} \subset (\epsilon^{1 - \zeta(j_k)}, \epsilon^{1 - \zeta_{+}(j_k)}) \cap \{8^j\}_{j \in \ZZ}$ defined as in \eqref{eq:GoodScales} with $n = j_k$, and with $\delta' \coloneqq \frac{\delta}{1 + 2 (A+1)}$ in place of $\delta$, to get that for some $\beta = \beta(S) > 0$, with probability $1 - O_{\epsilon}(\epsilon^{\beta})$ as $\epsilon \to 0$,
        \begin{align*}
          \locLFPP\left(B_{4 \epsilon^{1 - \zeta_{+}(j_N)}}(u), B_{4 \epsilon^{1 - \zeta_{+}(j_N)}}(v)\right) \
          &\leq \ C_N(\epsilon) D_h\left(u,v \right) \ \forall (u,v) \in \hyperlink{V}{V(R)}, \\
          D_h\left(B_{4 \epsilon^{1 - \zeta_{+}(j_N)}}(u), B_{4 \epsilon^{1 - \zeta_{+}(j_N)}}(v) \right) \
          &\leq \ C_N'(\epsilon) \locLFPP\left(u,v \right) \ \forall (u,v) \in \hyperlink{Vhat}{\hat{V}_{\epsilon}(R)},
        \end{align*}
        where $C_k(\epsilon)$ and $C_k'(\epsilon)$ are defined inductively by $C_0(\epsilon) \coloneqq C_0$, $C_0'(\epsilon) \coloneqq C_0$, and
        \begin{align*}
          C_{k+1}(\epsilon) \
          &\coloneqq \ \frac{A}{A+1} \left[\sup_{r \in \cR_{\epsilon}^{(k+1)}} \ratios \vee C_k(\epsilon) \right] + \left[\frac{1 + \delta'}{A+1} + 2 \delta' \right] \sup_{r \in \cR_{\epsilon}^{(k+1)}} \ratios, \\
          C_{k+1}'(\epsilon) \
          &\coloneqq \ \frac{A}{A+1} \left[\sup_{r \in \cR_{\epsilon}^{(k+1)}} \invratios \vee C_k'(\epsilon) \right] + \left[\frac{1 + \delta'}{A+1} + 2 \delta' \right] \sup_{r \in \cR_{\epsilon}^{(k+1)}} \invratios.
        \end{align*}

        Note that $C_N(\epsilon) \vee C_N'(\epsilon) \leq (1 + \delta)^2$ for all $\epsilon \in \cE_S$.
        Indeed, if for some $\epsilon \in \cE_S$ and some $k \in \{0,1,2,\ldots, N\}$, we have $C_k(\epsilon) \leq (1 + \delta)^2$, then
        \begin{align*}
          C_{k+1}(\epsilon) \
          &\leq \ \frac{A}{A+1} \left(1 + \delta \right)^2 + \left[\frac{1 + \delta'}{A+1} + 2 \delta' \right](1 + \delta) \\
          &= \ \frac{A}{A+1} \left(1 + \delta \right)^2 + \frac{1 + \delta}{A+1} + \delta' \frac{1 + 2 (A + 1)}{A + 1}(1 + \delta) \\
          &= \ \frac{A}{A+1} \left(1 + \delta \right)^2 + \frac{1 + \delta}{A+1} + \frac{\delta (1 + \delta)}{A+1} \\
          &= \ \left(1 + \delta \right)^2.
        \end{align*}
        and likewise $C_{k+1}'(\epsilon) \leq (1 + \delta)^2$.
        So it is enough to show that for every $\epsilon \in \cE_S$, there are some $k, \ell \in \{1,2,\ldots, N\}$ with $C_k(\epsilon) \vee C_{\ell}'(\epsilon) \leq (1 + \delta)^2$.
        If $C_k(\epsilon) \vee C_{\ell}'(\epsilon) \leq 1 + \delta$ for some $k,\ell \in \{1,2,\ldots,N\}$, then we can take these $k$ and $\ell$.
        Assuming $C_k(\epsilon) > 1 + \delta$ for all $k \in \{1,2, \ldots, N\}$,
        \begin{align*}
          C_N(\epsilon) \
          &= \ \left(\frac{A}{A+1} \right)^{N} C_0 + \sum_{k=0}^{N-1} \left(\frac{A}{A+1} \right)^k \left[\frac{1 + \delta'}{A+1} + 2 \delta' \right] \sup_{r \in \cR_{\epsilon}^{(k)}} \ratios \\
          &\leq \ \left(\frac{A}{A+1} \right)^N C_0 + \sum_{k=0}^{N-1} \left(\frac{A}{A+1} \right)^k \left[\frac{1 + \delta'}{A+1} + 2 \delta' \right] (1 + \delta) \\
          &\leq \ \delta^2 + (A + 1) \left[\frac{1 + \delta'}{A+1} + 2 \delta' \right](1 + \delta) \\
          &= \ \delta^2 + (1 + \delta) + \delta' \left(1 + 2 (A+1) \right) \\
          &= \ \left(1 + \delta \right)^2.
        \end{align*}
        Likewise, if $C_{\ell}'(\epsilon) > 1 + \delta$ for all $\ell \in \{1,2, \ldots, N\}$, then $C_N'(\epsilon) \leq (1 + \delta)^2$.

        Summarizing, we have shown that as $\epsilon \to 0$ along $\cE_S$, it holds with probability $1 - O_{\epsilon}(\epsilon^{\beta(S)})$ that
        \begin{align*}
          \locLFPP\left(B_{4 \epsilon^{1 - \zeta_{+}(j_N)}}(u), B_{4 \epsilon^{1 - \zeta_{+}(j_N)}}(v) \right) \
          &\leq \ \left(1 + \delta \right)^2 D_h\left(u,v \right) \ \forall (u,v) \in \hyperlink{V}{V(R)}, \\
          D_h\left(B_{4 \epsilon^{1 - \zeta_{+}(j_N)}}(u), B_{4 \epsilon^{1 - \zeta_{+}(j_N)}}(v) \right) \
          &\leq \ \left(1 + \delta \right)^2 \locLFPP\left(u,v \right) \ \forall (u,v) \in \hyperlink{Vhat}{\hat{V}_{\epsilon}(R)}.
        \end{align*}
        Repeat this for all finitely many $S \subset \{1,2,\ldots, 3N\}$ with $\# S \geq N$ and $\inf \cE_S = 0$, then let $\beta$ be the smallest of the $\beta(S)$ while letting $\zeta$ be the largest of the $\zeta_{+}(j_N)$ to get \eqref{eq:LipschitzConstant1LocalDh} and \eqref{eq:LipschitzConstant1DhLocal} with $(1 + \delta)^2$ in place of $1 + \delta$.
      \end{proof}

  \section{Proof of Theorem \ref{thm:AlmostSureConvergence}}
    \label{section:ProofOfMainTheorem}
    The goal of this section is to deduce Theorem \ref{thm:AlmostSureConvergence} from Proposition \ref{prop:LipschitzConstant1}.
    The strategy of the proof is roughly the following.
    By Borel-Cantelli applied to the sequence $(n^{-a})_{n=1}^{\infty}$ with $a$ large enough, \eqref{eq:LipschitzConstant1LocalDh} and \eqref{eq:LipschitzConstant1DhLocal} will almost surely hold for $\epsilon = n^{-a}$ for all $n$ sufficiently large.
    This is enough to show that for all $(u,v) \in \hyperlink{V}{V(R)} \cap \hyperlink{Vhat}{\hat{V}_{\epsilon}(R)}$,
    \begin{align}
      \left|\locLFPP[n^{-a}](u,v) - D_h(u,v) \right| \
      &\leq \ \delta \left( D_h(u,v) \vee \locLFPP[n^{-a}](u,v) \right) + o_{n}(1),
      \label{eq:UniformConvergenceIntuition}
    \end{align}
    where the $o_{n}(1)$ is the error from removing the $4 \epsilon^{1 - \zeta}$ balls in \eqref{eq:LipschitzConstant1LocalDh} and \eqref{eq:LipschitzConstant1DhLocal}.
    This error converges to $0$ in probability as $n \to \infty$, but we will argue that it actually converges to $0$ almost surely.
    We would like to now send $\delta \to 0$ and $R \to \infty$ to get almost sure convergence along the sequence $(n^{-a})_{n=1}^{\infty}$, then use a continuity argument to extend the convergence to the continuum index $\epsilon$.
    However, $a$ depends on $\delta$ (because the exponent $\beta$ in Proposition \ref{prop:LipschitzConstant1} depends on $\delta$), so we actually need to do the continuity argument first to show \eqref{eq:UniformConvergenceIntuition} holds along the continuum index $\epsilon$, then send $\delta \to 0$ and $R \to \infty$.

    Before making this reasoning precise, we will need a few more lemmas.
    The first two are the continuity estimates needed to pass from the discrete index $n^{-a}$ to the continuum.

    \begin{lemma}
      \label{lemma:ScalingConstantsContinuity}
      Fix $a > 0$.
      For each $\epsilon \in (0,1)$, let $n = n(\epsilon)$ be the unique integer with $\epsilon \in \hlint{(n+1)^{-a}}{n^{-a}}$.
      Then
      \begin{align*}
        \lim_{\epsilon \to 0} \frac{\fa_{\epsilon}}{\fa_{n(\epsilon)^{-a}}} \
        &= \ 1.
      \end{align*}

      \begin{proof}
        By \cite[Corollary 1.11]{ExistenceAndUniqueness}, $\epsilon \mapsto \fa_{\epsilon}$ is regularly varying with exponent $1 - \xi Q$.
        So $\lim_{\epsilon \to 0} \frac{\fa_{C \epsilon}}{\fa_{\epsilon}} = C^{1 - \xi Q}$ uniformly over all $C$ in the interval $[1/2, 2]$.
        Since $\frac{\epsilon}{n(\epsilon)^{-a}} \in [1/2, 2]$ for all $\epsilon$ sufficiently small,
        \begin{align*}
          \left|\frac{\fa_{\epsilon}}{\fa_{n(\epsilon)^{-a}}} - 1 \right| \
          &\leq \ \left|\frac{\fa_{\frac{\epsilon}{n(\epsilon)^{-a}} n(\epsilon)^{-a}}}{\fa_{n(\epsilon)^{-a}}} - \left(\frac{\epsilon}{n(\epsilon)^{-a}} \right)^{1 - \xi Q} \right| + \left|\left(\frac{\epsilon}{n(\epsilon)^{-a}} \right)^{1 - \xi Q} - 1 \right| \
          \to \ 0.
        \end{align*}
      \end{proof}
    \end{lemma}
  
    \begin{lemma}
      \label{lemma:ContinuityHeatKernelMollification}
      Fix $a > 0$ and a compact set $K \subset \CC$.
      Almost surely,
      \begin{align*}
        \lim_{n \to \infty} \sup_{z \in K} \sup_{\epsilon, \delta \in [(n+1)^{-a}, n^{-a}]} \left|h_{\epsilon}^{*}(z) - h_{\delta}^{*}(z) \right| \vee \left|\hat{h}_{\epsilon}^{*}(z) - \hat{h}_{\delta}^{*}(z) \right| \
        &= \ 0.
      \end{align*}
  
      \begin{proof}
        It will suffice by Lemma \ref{lemma:PropertiesOfLocalizedFieldAndLFPP} to prove 
        \begin{align*}
          \lim_{n \to \infty} \sup_{z \in K} \sup_{\epsilon, \delta \in [(n+1)^{-a}, n^{-a}]} \left|h_{\epsilon}^{*}(z) - h_{\delta}^{*}(z) \right| \
          &= \ 0.
        \end{align*}
        Write $h_{\epsilon}^{*}(z)$ and $h_{\delta}^{*}(z)$ in polar coordinates, then use the Cauchy-Schwarz inequality to get
        \begin{align*}
          \left|h_{\epsilon}^{*}(z) - h_{\delta}^{*}(z) \right| \
          &\leq \ \int\limits_{0}^{\infty} r |h_r(z)| \left|\frac{2}{\epsilon^2} e^{-r^2/\epsilon^2} - \frac{2}{\delta^2} e^{-r^2/\delta^2} \right| \, \D r \\
          &= \ \frac{2}{\delta^2} \int\limits_{0}^{\infty} r |h_r(z)| e^{-r^2/\delta^2} \left|\frac{\delta^2}{\epsilon^2} \left(e^{-r^2/\delta^2} \right)^{\delta^2/\epsilon^2 - 1} - 1 \right| \, \D r \\
          &\leq \ \frac{2}{\delta^2} \left(\int_{0}^{\infty} r^2 |h_r(z)|^2 e^{-r^2/\delta^2} \, \D r \right)^{1/2} \\
          &\qquad \qquad \cdot \left(\int_{0}^{\infty} e^{-r^2/\delta^2} \left|\frac{\delta^2}{\epsilon^2} \left(e^{-r^2/\delta^2} \right)^{\delta^2/\epsilon^2 - 1} - 1 \right|^2 \, \D r \right)^{1/2}.
        \end{align*}
        For the first integral, use the bound $|h_r(z)| \leq C \max\{\log(1/r), \log r, 1\}$ for all $z \in K$ and all $r > 0$, where $C$ is some $K$-dependent random variable (see \cite[Lemma 2.2]{WeakLQGMetrics}), to get
        \begin{align*}
          \int\limits_{0}^{\infty} r^2 |h_r(z)|^2 e^{-r^2/\delta^2} \, \D r \
          &\leq \ C^2 \int\limits_{0}^{1} r^2(1 - \log r)^2 e^{-r^2/\delta^2} \, \D r + C^2 \int\limits_{1}^{\infty} r^2 (1 + \log r)^2 e^{-r^2/\delta^2} \, \D r \\
          &= \ C^2 \int\limits_{0}^{1/\delta} \delta^3 u^2 \left(1 - \log \delta - \log u \right)^2 e^{-u^2} \, \D u \\
          &\qquad \qquad + C^2 \int\limits_{1/\delta}^{\infty} \delta^3 u^2 (1 + \log \delta + \log u)^2 e^{-u^2} \, \D u \\
          &\leq \ C \delta^3 \left(\log \delta^{-1} \right)^2
        \end{align*}
        where we have absorbed constants in to $C$.
        For the second integral, change variables $r = \delta u$ to get
        \begin{align*}
          \int\limits_{0}^{\infty} e^{-r^2/\delta} \left|\frac{\delta^2}{\epsilon^2} \left(e^{-r^2/\delta^2} \right)^{\delta^2/\epsilon^2 - 1} - 1 \right|^2 \, \D r \
          &= \ \delta \int\limits_{0}^{\infty} \left(\frac{\delta^4}{\epsilon^4} e^{-u^2(2 \delta^2/\epsilon^2 - 1)} - 2 \frac{\delta^2}{\epsilon^2} e^{-u^2 \delta^2/\epsilon^2} + e^{-u^2} \right) \, \D u \\
          &= \ \frac{\sqrt{\pi}}{2} \delta \left(\frac{\delta^4}{\epsilon^4} \frac{1}{\sqrt{2 \delta^2/\epsilon^2 - 1}} - 2 \frac{\delta}{\epsilon} + 1 \right).
        \end{align*}
        Writing $s = \delta/\epsilon$, Taylor expanding around $s = 1$ shows that there is a constant $c > 0$ such that if, say, $s \in [0.9, 1.1]$,
        \begin{align*}
          \frac{\delta^4}{\epsilon^4} \frac{1}{\sqrt{2 \delta^2/\epsilon^2 - 1}} - 2 \frac{\delta}{\epsilon} + 1 \
          &= \ \frac{s^4}{\sqrt{2 s^2 - 1}} - 2 s + 1 \\
          &\leq \ \left(1 + 2 (s-1) + c (s - 1)^2 \right) - 2s + 1 \\
          &= \ c (s - 1)^2.
        \end{align*}
        Combining, we obtain
        \begin{align*}
          \left|h_{\epsilon}^{*}(z) - h_{\delta}^{*}(z) \right| \
          &\leq \ C \log(\delta^{-1}) \left(\frac{\delta}{\epsilon} - 1 \right)
        \end{align*}
        for another random variable $C$.
        When $\delta, \epsilon \in [(n+1)^{-a}, n^{-a}]$ with $\epsilon \leq \delta$, this becomes 
        \begin{align*}
          \left|h_{\epsilon}^{*}(z) - h_{\delta}^{*}(z) \right| \
          &\leq \ C a \log(n+1) \left(\left(\frac{n+1}{n} \right)^a - 1 \right),
        \end{align*}
        which converges to $0$ as $n \to \infty$.
      \end{proof}
    \end{lemma}

    The next lemma shows that the errors from replacing balls by points in \eqref{eq:LipschitzConstant1LocalDh} and \eqref{eq:LipschitzConstant1DhLocal} converges to $0$ almost surely.
    For our purposes, we will apply this with $1 - \zeta$ very close to $0$, in which case the trivial bound for the $\LFPP$- or $\locLFPP$-length of a line segment isn't sufficient.

    \begin{lemma}
      \label{lemma:LengthsOfSmallSegments}
      Fix $\zeta \in (0,1)$.
      There exists $a_{*} = a_{*}(\zeta) > 0$ such that for each $a \geq a_{*}$ and each compact $K \subset \CC$, almost surely
      \begin{align}
        \lim_{n \to \infty} \sup_{\substack{z,w \in K \\ |z - w| \leq 4 n^{-a(1 - \zeta)}}} \LFPP[n^{-a}](z,w) \vee \locLFPP[n^{-a}](z,w) \
        &= \ 0. \label{eq:LengthsOfSmallSegments}
      \end{align}
      
      \begin{proof}
        The idea of the proof is that if $1 - \zeta$ is very close to $1$, then \eqref{eq:LengthsOfSmallSegments} will hold by a trivial upper bound on the $\LFPP$-length or $\locLFPP$-length of a line segment of Euclidean length $\leq 4 n^{-a(1 - \zeta)}$.
        However, the claim of Lemma \ref{lemma:LengthsOfSmallSegments} is that this holds even when $1 - \zeta$ is close to $0$.
        To obtain the latter, we can use Proposition \ref{prop:UpToConstants} with some $\zeta'$ close to $0$ to upper bound the left-hand side of \eqref{eq:LengthsOfSmallSegments} by the $D_h$-length of a segment of Euclidean length $\leq 4 n^{-a(1 - \zeta)}$, plus the left-hand side of \eqref{eq:LengthsOfSmallSegments} with $\zeta'$ in place of $\zeta$.
        As long as $\zeta'$ is close enough to $0$ for the trivial estimate to hold, Lemma \ref{lemma:LengthsOfSmallSegments} will follow from the continuity of $D_h$.

        Fix $\eta \in (0,1)$ and $R > 0$.
        We will first prove that there is a random variable $C > 0$ such that almost surely, for all sufficiently small $\epsilon$ and all $z \in B_R(0)$,
        \begin{align}
          \left|h_{\epsilon}^{*}(z) \right| \vee \left|\hat{h}_{\epsilon}^{*}(z) \right| \
          &\leq \ (1 + \eta)(2 + \eta) \log \epsilon^{-1} + C.
          \label{eq:TruncatedHeatKernelMollificationBound}
        \end{align}
        Indeed, elementary estimates for the circle average process (see \cite[Lemma 2.3]{WeakLQGMetrics} and \cite[Lemma 3.1]{ThickPoints}) imply
        \begin{align*}
          \lim_{r \to \infty} \sup_{z \in B_R(0)} \frac{|h_r(z)|}{\log r} \
          &= \ 0, \\
          \limsup_{r \to 0} \sup_{z \in B_R(0)} \frac{|h_r(z)|}{\log(1/r)} \
          &\leq \ 2.
        \end{align*}
        Writing $Z_{\epsilon} \hat{h}_{\epsilon}^{*}(z)$ in polar coordinates, we get for all $z \in B_R(0)$,
        \begin{align*}
          \left|Z_{\epsilon} \hat{h}_{\epsilon}^{*}(z) \right| \
          &\leq \ \frac{2}{\epsilon^2} \int\limits_{0}^{\infty} r |h_r(z)| \psi_{\epsilon}(r) e^{-r^2/\epsilon^2} \, \D r \\
          &\leq \ \frac{2(2 + \eta)}{\epsilon^2} \int\limits_{0}^{r_1} r \log(1/r) e^{-r^2/\epsilon^2} \, \D r + \frac{2 C}{\epsilon^2} \int\limits_{r_1}^{r_2} r e^{-r^2/\epsilon^2} \, \D r + \frac{2}{\epsilon^2} \int\limits_{r_2}^{\infty} r \log r e^{-r^2/\epsilon^2} \, \D r \\
          &\leq \ (2 + \eta) \log(1/\epsilon) + C
        \end{align*}
        for random variables $0 < r_1 < 1 < r_2 < \infty$ and $C > 0$.
        A similar calculation shows $|h_{\epsilon}^{*}(z)| \leq (2 + \eta) \log(1/\epsilon) + C$.
        Since $Z_{\epsilon} \leq \int_{\CC} p_{\epsilon^2/2}(w) \, \D w = 1$ and 
        \begin{align*}
          1 - Z_{\epsilon} \
          &= \ \frac{2}{\epsilon^2} \int\limits_{0}^{\infty} r (1 - \psi_{\epsilon}(r)) e^{-r^2/\epsilon^2} \, \D r \
          \leq \ \frac{2}{\epsilon^2} \int\limits_{\frepsilon/2}^{\infty} r e^{-r^2/\epsilon^2} \, \D r \
          = \ e^{- [\log \epsilon^{-1}]^2/4},
        \end{align*}
        it follows that 
        \begin{align*}
          \left|\hat{h}_{\epsilon}^{*}(z) \right| \
          &\leq \ \frac{1}{1 - e^{-[\log \epsilon^{-1}]^{2}/4}} (2 + \eta) \log(1/\epsilon) + \frac{C}{1 - e^{-[\log \epsilon^{-1}]^{2}/4}}.
        \end{align*}
        Equation \eqref{eq:TruncatedHeatKernelMollificationBound} now follows for all $\epsilon$ small enough that $1 - e^{-[\log \epsilon^{-1}]^{2}/4} > (1 + \eta)^{-1}$.
  
        Now to prove \eqref{eq:LengthsOfSmallSegments}, it will suffice to prove that there exists $a_{*}$ such that for each $a \geq a_{*}$ and each $R > 0$, almost surely
        \begin{align*}
          \lim_{n \to \infty} \sup_{\substack{z,w \in B_R(0) \\ |z - w| \leq 4 n^{-a(1 - \zeta)}}} \LFPP[n^{-a}](z,w) \vee \locLFPP[n^{-a}](z,w) \
          &= \ 0.
        \end{align*}
        Equation \eqref{eq:TruncatedHeatKernelMollificationBound} and \cite[Theorem 1.11]{UpToConstants} imply there exist $b > 0$ and random $C > 0$ that for each $\epsilon$ sufficiently small and for all $z,w \in B_R(0)$,
        \begin{align*}
          \LFPP(z,w) \vee \locLFPP(z,w) \
          &\leq \ C |z - w| \epsilon^{-\xi(1 + \eta)(2 + \eta)} \epsilon^{\xi Q - 1} \left(\log \epsilon^{-1} \right)^b
        \end{align*}
        Choose $\eta, \zeta' \in (0,1)$ such that $\zeta' < \zeta$ and $\xi(Q - (1 + \eta)(2 + \eta)) - \zeta' > 0$ (here, we are using the fact that $Q > 2$), so then almost surely, 
        \begin{align*}
          \sup_{\substack{z,w \in B_R(0) \\ |z - w| \leq 4 \epsilon^{1 - \zeta'}}} \LFPP(z,w) \vee \locLFPP(z,w) \
          &\leq \ C \epsilon^{\xi (Q - (1 + \eta)(2 + \eta)) - \zeta'} \left(\log \epsilon^{-1} \right)^b \
          \to \ 0.
        \end{align*}
  
        Finally, note that
        \begin{align}
          &\sup_{\substack{z,w \in B_R(0) \\ |z - w| \leq 4 \epsilon^{1 - \zeta}}} \LFPP(z,w) \vee \locLFPP(z,w) \nonumber \\
          &\qquad \leq \ 2 \sup_{\substack{z,w \in B_{R + 1}(0) \\ |z - w| \leq 4 \epsilon^{1 - \zeta'}}} \LFPP(z,w) \vee \locLFPP(z,w) \label{eq:ErrorTermsGoTo0} \\
          &\qquad \qquad + \sup_{\substack{z,w \in B_R(0) \\ 4 \epsilon^{1 - \zeta'} < |z - w| \leq 4 \epsilon^{1 - \zeta}}} \LFPP\left(B_{\epsilon^{1 - \zeta'}}(z), B_{\epsilon^{1 - \zeta'}}(w) \right) \vee \locLFPP\left(B_{\epsilon^{1 - \zeta'}}(z), B_{\epsilon^{1 - \zeta'}}(w) \right). \nonumber
        \end{align}
        The first term on the right-hand side of \eqref{eq:ErrorTermsGoTo0} converges to $0$ almost surely by the preceding argument.
        For the second term, use Proposition \ref{prop:UpToConstants} to find $C_0, \beta > 0$ depending only on $\zeta'$ such that with probability $1 - O(\epsilon^{\beta})$ as $\epsilon \to 0$, 
        \begin{align*}
          \LFPP\left(B_{\epsilon^{1 - \zeta'}}(z), B_{\epsilon^{1 - \zeta'}}(w) \right) \
          &\leq \ C_0 D_h\left(z,w; B_{R+1}(0) \right) \ 
          \forall z,w \in B_{R+1}(0), \\
          \locLFPP\left(B_{\epsilon^{1 - \zeta'}}(z), B_{\epsilon^{1 - \zeta'}}(w) \right) \
          &\leq \ C_0 D_h\left(z,w; B_{R+1}(0) \right) \ 
          \forall z,w \in B_{R+1}(0).
        \end{align*}
        Choose $a_{*}$ large enough that $a_{*} \beta > 1$, so then if $a \geq a_{*}$, the Borel-Cantelli Lemma implies that almost surely, for all $n$ sufficiently large,
        \begin{align*}
          \locLFPP[n^{-a}]\left(B_{n^{-a(1 - \zeta')}}(z), B_{n^{-a(1 - \zeta')}}(w) \right) \
          &\leq \ C_0 D_h\left(z,w; B_{R+1}(0) \right) \ \forall z,w \in B_{R+1}(0), \\
          \LFPP[n^{-a}]\left(B_{n^{-a(1 - \zeta')}}(z), B_{n^{-a(1 - \zeta')}}(w) \right) \
          &\leq \ C_0 D_h\left(z,w; B_{R+1}(0) \right) \ \forall z,w \in B_{R+1}(0),
        \end{align*}
        and hence that the second term on the right-hand side of \eqref{eq:ErrorTermsGoTo0} with $\epsilon = n^{-a}$ is bounded above by
        \begin{align*}
          \sup_{\substack{z,w \in B_{R}(0) \\ |z - w| \leq 4 n^{-a(1 - \zeta)}}} C_0 D_h\left(z,w; B_{R+1}(0) \right).
        \end{align*}
        By uniform continuity of $D_h(\cdot, \cdot; B_{R+1}(0))$ on $\overline{B_{R}(0)} \times \overline{B_{R}(0)}$, this converges to $0$ almost surely as $n \to \infty$.
      \end{proof}
    \end{lemma}

    For our final lemma, it is an easy to see that for a fixed compact set $K \subset \RR$, $P\{K^2 \subset \bigcup_{R>0} \hyperlink{Vhat}{\hat{V}_{\epsilon}(R)}\} = 1$ for all $\epsilon$.
    When we send $R \to \infty$ in the proof of Theorem \ref{thm:AlmostSureConvergence}, we would like for this event to not depend on $\epsilon$.
    In analogy with the definitions of $\hyperlink{V}{V(R)}$ and $\hyperlink{Vhat}{\hat{V}_{\epsilon}(R)}$, define
    \begin{align*}
      \hypertarget{Veps}{V_{\epsilon}(R)} \
      &\coloneqq \ \left\{(z,w) \in \CC^2 : \exists D_h^{\epsilon}\text{-geodesic from } z \text{ to } w \text{ contained in } B_R(0) \right\}.
    \end{align*}

    \begin{lemma}
      \label{lemma:NoInternalMetrics}
      Fix a compact set $K \subset \CC$.
      Almost surely, there exists $R_0 > 0$ and $\epsilon_0 \in (0,1)$ such that for every $R \geq R_0$ and every $\epsilon \in (0, \epsilon_0)$, $K^2 \subset \hyperlink{V}{V(R)} \cap \hyperlink{Veps}{V_{\epsilon}(R)} \cap \hyperlink{Vhat}{\hat{V}_{\epsilon}(R)}$.

      \begin{proof}
        The result for $\hyperlink{V}{V(R)}$ holds by \cite[Lemma 3.8]{WeakLQGMetrics}, so we just need to prove the result for $\hyperlink{Veps}{V_{\epsilon}(R)}$ and $\hyperlink{Vhat}{\hat{V}_{\epsilon}(R)}$.
        The idea is to use Borel-Cantelli to show that almost surely, \eqref{eq:UpToConstantsLocDh}, \eqref{eq:UpToConstantsDhLoc}, \eqref{eq:UpToConstantsLFPPDh}, and \eqref{eq:UpToConstantsDhLFPP} hold with $\epsilon = n^{-a}$ for all $n$ large enough, then use a continuity estimate to show that the same holds for all $\epsilon$ sufficiently small.
        Then the result follows from the fact that $K^2 \subset \hyperlink{V}{V(R)}$ for all $R$ sufficiently large.

        Fix $p, \zeta \in (0,1)$.
        Choose $\beta, C_0 > 0$ according to Proposition \ref{prop:UpToConstants}.
        Choose $R > 1$ such that $\overline{B_1(K)} \subset B_{R-1}(0)$, and the event $A$ that 
        \begin{align}
          \sup_{u,v \in K} D_h(u,v) \
          &< \ C_0^{-2} D_h\left(\overline{B_1(K)}, \partial B_{R-1}(0) \right) - C_0^{-1}
          \label{eq:NoInternalMetricsEventA}
        \end{align}
        occurs with probability at least $p$.
        Choose $a > 0$ according to Lemma \ref{lemma:LengthsOfSmallSegments} such that $a \beta > 1$.
        Let $H_{\epsilon}$ be the event that \eqref{eq:UpToConstantsLocDh}, \eqref{eq:UpToConstantsDhLoc}, \eqref{eq:UpToConstantsLFPPDh}, and \eqref{eq:UpToConstantsDhLFPP} all hold with $B_{R+1}(0)$ in place of $U$.
        Then on $A \cap \bigcup_{k=1}^{\infty} \bigcap_{n=k}^{\infty} H_{n^{-a}}$, for all $n$ sufficiently large,
        \begin{align*}
          \sup_{u,v \in K} \locLFPP[n^{-a}](u,v)
          &\leq \sup_{u,v \in K} \locLFPP[n^{-a}]\left(B_{4n^{-a(1 - \zeta)}}(u), B_{4n^{-a(1 - \zeta)}}(v) \right) \\
          &\qquad \qquad + 2 \sup_{u \in K} \sup_{v \in \partial B_{4n^{-a(1 - \zeta)}}(u)} \locLFPP[n^{-a}](u, v) \\
          &\leq C_0 \sup_{u,v \in K} D_h\left(u,v \right) + 2 \sup_{u \in K} \sup_{v \in \partial B_{4 n^{-a(1 - \zeta)}}(u)} \locLFPP[n^{-a}](u,v) & (\text{by \eqref{eq:UpToConstantsLocDh}}) \\
          &\leq C_0^{-1} D_h\left(\overline{B_1(K)}, \partial B_{R-1}(0) \right) & (\text{by \eqref{eq:NoInternalMetricsEventA}}) \\
          &\qquad \qquad - 1 + 2 \sup_{u \in K} \sup_{v \in \partial B_{4 n^{-a(1 - \zeta)}}(u)} \locLFPP[n^{-a}](u,v) \\
          &\leq C_0^{-1} D_h\left(B_{n^{-a(1 - \zeta)}}(K), B_{n^{-a(1 - \zeta)}}(\partial B_R(0)) \right) \\
          &\qquad \qquad - 1 + 2 \sup_{u \in K} \sup_{v \in \partial B_{4 n^{-a(1 - \zeta)}}(u)} \locLFPP[n^{-a}](u,v) \\
          &\leq \locLFPP[n^{-a}]\left(K, \partial B_R(0) \right) & (\text{by \eqref{eq:UpToConstantsDhLoc}}) \\
          &\qquad \qquad - 1 + 2 \sup_{u \in K} \sup_{v \in \partial B_{4 n^{-a(1 - \zeta)}}(u)} \locLFPP[n^{-a}](u,v),
        \end{align*}
        and likewise with $D_h^{\epsilon}$ in place of $\hat{D}_h^{\epsilon}$.
        By Lemma \ref{lemma:LengthsOfSmallSegments}, the event $A'$ that the last term converges to $0$ occurs with probability $1$.
        So on $A \cap A' \cap \bigcup_{k=1}^{\infty} \bigcap_{n=k}^{\infty} H_{n^{-a}}$, for all $n$ sufficiently large,
        \begin{align*}
          \sup_{u,v \in K} \locLFPP[n^{-a}](u,v) \
          &< \ \locLFPP[n^{-a}]\left(K, \partial B_R(0) \right) - \frac{1}{2}, 
        \end{align*}
        and likewise with $D_h^{\epsilon}$ in place of $\hat{D}_h^{\epsilon}$.
        In particular, $K^2 \subset V_{n^{-a}}(R) \cap \hyperlink{Vhat}{\hat{V}_{n^{-a}}(R)}$, so by Lemma \ref{lemma:ContinuityHeatKernelMollification}, almost surely on $A \cap A' \cap \bigcup_{k=1}^{\infty} \bigcap_{n=k}^{\infty} H_{n^{-a}}$, for all $n$ sufficiently large and all $\epsilon \in [(n+1)^{-a}, n^{-a}]$,
        \begin{align*}
          \sup_{u,v \in K} \locLFPP(u,v) \
          &\leq \ e^{o_{\epsilon}(1)} \sup_{u,v \in K} \locLFPP[n^{-a}](u,v), \\
          \locLFPP\left(K, \partial B_R(0) \right) \
          &\geq \ e^{-o_{\epsilon}(1)} \locLFPP[n^{-a}](K, \partial B_R(0)) 
        \end{align*}
        for some $R$-dependent random error $o_{\epsilon}(1)$ tending to $0$ almost surely as $\epsilon \to 0$.
        The same is true with $D_h^{\epsilon}$ in place of $\hat{D}_h^{\epsilon}$.
        It follows that $K^2 \subset \hyperlink{Veps}{V_{\epsilon}(R)} \cap \hyperlink{Vhat}{\hat{V}_{\epsilon}(R)}$ for all $\epsilon$ sufficiently small.
        Sending $p \to 1$ proves the claim.
      \end{proof}
    \end{lemma}

    \begin{proof}[Proof of Theorem \ref{thm:AlmostSureConvergence}]
      Fix $\delta \in (0,1)$, choose $\zeta$ and $\beta$ according to Proposition \ref{prop:LipschitzConstant1}, and let $H_{\epsilon}$ be the event that \eqref{eq:LipschitzConstant1LocalDh} and \eqref{eq:LipschitzConstant1DhLocal} hold.
      Choose $a > a_{*} \vee 1/\beta$, where $a_{*}(\zeta)$ is as in Lemma \ref{lemma:LengthsOfSmallSegments}.

      Now fix a compact set $K \subset \CC$ and a parameter $R > 0$ which will eventually be sent to $\infty$.
      First note that on $\bigcup_{k=1}^{\infty} \bigcap_{n=k}^{\infty} H_{n^{-a}}$, 
      \begin{align*}
        \limsup_{n \to \infty} \sup_{(u,v) \in K^2 \cap \hyperlink{V}{V(R)}} \locLFPP[n^{-a}](u,v)
        &\leq 2 \limsup_{n \to \infty} \sup_{z \in K} \sup_{\substack{w \in \CC \\ |w - z| \leq 4 n^{-a(1 - \zeta)}}} \locLFPP[n^{-a}](z,w) \\
        &\qquad \qquad + (1 + \delta) \sup_{u,v \in K} D_h(u,v).
      \end{align*}
      The first term on the right-hand side is $0$ by Lemma \ref{lemma:LengthsOfSmallSegments}, so 
      \begin{align}
        \limsup_{n \to \infty} \sup_{(u,v) \in K^2 \cap \hyperlink{V}{V(R)}} \locLFPP[n^{-a}](u,v) \
        &\leq \ (1 + \delta) \sup_{u,v \in K} D_h(u,v) \
        < \ \infty \text{ a.s.}
        \label{eq:LimsupAlongDiscreteIndexFinite}
      \end{align}
      Now for each $\epsilon \in (0,1)$, write $n = n(\epsilon)$ for the integer with $\epsilon \in \hlint{(n+1)^{-a}}{n^{-a}}$.
      Lemma \ref{lemma:ContinuityHeatKernelMollification} implies that for all $u,v \in K$ with $u \neq v$ and $(u,v) \in \hyperlink{Vhat}{\hat{V}_{\epsilon}(R)} \cap \hyperlink{Vhat}{\hat{V}_{n^{-a}}(R)}$,
      \begin{align}
        \frac{\fa_{\epsilon}}{\fa_{n^{-a}}} e^{-o_{\epsilon}(1)} \
        &\leq \ \frac{\locLFPP(u,v)}{\locLFPP[n^{-a}](u,v)} \
        \leq \ \frac{\fa_{\epsilon}}{\fa_{n^{-a}}} e^{o_{\epsilon}(1)},
        \label{eq:ContinuityLocalizedLFPP}
      \end{align}
      where the $o_{\epsilon}(1)$ converges to $0$ almost surely as $\epsilon \to 0$.
      
      Now assume $\bigcup_{k=1}^{\infty} \bigcap_{m=k}^{\infty} H_{m^{-a}}$ occurs, so 
      \begin{align*}
        &\limsup_{\epsilon \to 0} \sup_{(u,v) \in K^2 \cap \hyperlink{V}{V(R)} \cap \hyperlink{Vhat}{\hat{V}_{\epsilon}(R)} \cap \hyperlink{Vhat}{\hat{V}_{n^{-a}}(R)}} \left[\locLFPP(u,v) - D_h(u,v) \right] \\
        &\qquad \qquad \leq \ \limsup_{\epsilon \to 0} \sup_{(u,v) \in K^2 \cap \hyperlink{V}{V(R)} \cap \hyperlink{Vhat}{\hat{V}_{\epsilon}(R)} \cap \hyperlink{Vhat}{\hat{V}_{n^{-a}}(R)}} \left[\locLFPP(u,v) - \locLFPP[n^{-a}](u,v) \right] \\
        &\qquad \qquad \qquad \qquad + 2 \limsup_{\epsilon \to 0} \sup_{z \in K} \sup_{\substack{w \in \CC \\ |w - z| \leq  \in 4 n^{-a(1 - \zeta)}}} \locLFPP[n^{-a}](u,v) + \delta \sup_{u,v \in K} D_h(u,v).
      \end{align*}
      The second term on the right-hand side is $0$ by Lemma \ref{lemma:LengthsOfSmallSegments}.
      To see that the first is $0$ almost surely, combine \eqref{eq:LimsupAlongDiscreteIndexFinite}, \eqref{eq:ContinuityLocalizedLFPP}, and Lemma \ref{lemma:ScalingConstantsContinuity} to see that
      \begin{align*}
        &\limsup_{\epsilon \to 0} \sup_{(u,v) \in K^2 \cap \hyperlink{V}{V(R)} \cap \hyperlink{Vhat}{\hat{V}_{\epsilon}(R)} \cap \hyperlink{Vhat}{\hat{V}_{n^{-a}}(R)}} \left[\locLFPP(u,v) - \locLFPP[n^{-a}](u,v) \right] \\
        &\qquad \leq \ \lim_{\epsilon \to 0} \sup_{\substack{(u,v) \in K^2 \cap \hyperlink{Vhat}{\hat{V}_{\epsilon}(R)} \cap \hyperlink{Vhat}{\hat{V}_{n^{-a}}(R)} \\ u \neq v}} \left|\frac{\locLFPP(u,v)}{\locLFPP[n^{-a}](u,v)} - 1 \right| (1 + \delta) \sup_{u,v \in K} D_h(u,v) \
        = \ 0.
      \end{align*}
      Thus, we get almost surely on $\bigcup_{k=1}^{\infty} \bigcap_{m=k}^{\infty} H_{m^{-a}}$,
      \begin{align*}
        \limsup_{\epsilon \to 0} \sup_{(u,v) \in K^2 \cap \hyperlink{V}{V(R)} \cap \hyperlink{Vhat}{\hat{V}_{\epsilon}(R)} \cap \hyperlink{Vhat}{\hat{V}_{n^{-a}}(R)}} \left[\locLFPP(u,v) - D_h(u,v) \right] \
        &\leq \ \delta \sup_{u,v \in K} D_h(u,v).
      \end{align*}
      By Lemma \ref{lemma:NoInternalMetrics}, we can send $R \to \infty$ and then $\delta \to 0$ to get 
      \begin{align*}
        \limsup_{\epsilon \to 0} \sup_{u,v \in K} \left[\locLFPP(u,v) - D_h(u,v) \right] \
        &\leq \ 0 \text{ a.s.}
      \end{align*}
      The lower bound that
      \begin{align*}
        \liminf_{\epsilon \to 0} \inf_{u,v \in K} \left[\locLFPP(u,v) - D_h(u,v) \right] \
        &\geq \ 0 \text{ a.s.}
      \end{align*}
      is done analogously.

      We have now shown that $\locLFPP \to D_h$ almost surely, and it remains to prove that $\LFPP \to D_h$ almost surely.
      For this, choose an increasing collection of connected bounded open sets $(U_k)_{k=1}^{\infty}$ such that $\overline{U}_k \subset U_{k+1}$ for all $k$ and $\CC = \bigcup_{k=1}^{\infty} U_k$.
      By Lemma \ref{lemma:PropertiesOfLocalizedFieldAndLFPP}, the event
      \begin{align*}
        A \
        &\coloneqq \ \bigcap_{k=1}^{\infty} \left\{\lim_{\epsilon \to 0} \frac{\hat{D}_h^{\epsilon}(z,w; U_k)}{D_h^{\epsilon}(z,w;U_k)} = 1 \text{ uniformly over all } z,w \in U_k \text{ with } z \neq w \right\}
      \end{align*}
      occurs with probability $1$.
      Now fix a compact set $K \subset \CC$, so
      \begin{align*}
        &\limsup_{\epsilon \to 0} \sup_{u,v \in K} \left|\LFPP(u,v) - D_h(u,v) \right| \\
        &\qquad \qquad \leq \ \limsup_{\epsilon \to 0} \sup_{\substack{u,v \in K \\ u \neq v}} \left|\frac{D_h^{\epsilon}(u,v)}{\hat{D}_h^{\epsilon}(u,v)} - 1 \right| \limsup_{\epsilon \to 0} \sup_{u,v \in K} \locLFPP(u,v) \\
        &\qquad \qquad \qquad \qquad + \limsup_{\epsilon \to 0} \sup_{u,v \in K} \left|\locLFPP(u,v) - D_h(u,v)\right|.
      \end{align*}
      The second term is $0$ almost surely, while $\limsup_{\epsilon \to 0} \sup_{u,v \in K} \locLFPP(u,v) = \sup_{u,v \in K} D_h(u,v)$ is almost surely finite.
      By Lemma \ref{lemma:NoInternalMetrics}, almost surely there exists $R > 0$ and $\epsilon_0 > 0$ such that $K^2 \subset \hyperlink{Veps}{V_{\epsilon}(R)} \cap \hyperlink{Vhat}{\hat{V}_{\epsilon}(R)}$ for all $\epsilon \in (0, \epsilon_0)$. 
      It follows that with probability one,
      \begin{align*}
        \limsup_{\epsilon \to 0} \sup_{\substack{u,v \in K \\ u \neq v}} \left|\frac{D_h^{\epsilon}(u,v)}{\hat{D}_h^{\epsilon}(u,v)} - 1 \right| \
        &= \ 0.
      \end{align*}
    \end{proof}

    \subsection{LFPP Scaling Constants}
      We now turn to the proof of Theorem \ref{thm:LFPPScalingConstants}.

      \begin{proof}[Proof of Theorem \ref{thm:LFPPScalingConstants}]
        Fix $\delta \in (0,1)$ and we will prove the claim for $b \coloneqq \log(1 + \delta)$.
        Since $\epsilon \mapsto \fa_{\epsilon}$ is bounded above and below by positive $\xi$-dependent constants on any given compact subinterval of $(0, \infty)$, it will suffice to find $\tilde{\epsilon} \in (0,1)$ and $C = C(\xi, b) > 0$ such that \eqref{eq:LFPPScalingConstantsGoal} holds for all $\epsilon \in (0, \tilde{\epsilon})$.
        The idea of the proof is to choose $\tilde{\epsilon}$ to be $\epsilon_N$ from Lemma \ref{lemma:GoodRatios}, so then for each $\epsilon \in (0, \tilde{\epsilon})$, there are many values of $r$ for which 
        \begin{align}
          \left(1 + \delta \right)^{-1} r^{\xi Q - 1} \
          &\leq \ \frac{\fa_{\epsilon/r}}{\fa_{\epsilon}} \
          \leq \ \left(1 + \delta \right) r^{\xi Q - 1}.
          \label{eq:LFPPScalingConstantsStrategy}
        \end{align}
        Starting with an arbitrary $\epsilon_0 \in (0, \tilde{\epsilon})$, we will choose $r_1$. such that \eqref{eq:LFPPScalingConstantsStrategy} holds with $\epsilon = \epsilon_0$ and $r = r_1$.
        Then if $\epsilon_1 \coloneqq \epsilon_0/r_1$ is still smaller than $\tilde{\epsilon}$, we will choose $r_2$ so that \eqref{eq:LFPPScalingConstantsStrategy} holds with $\epsilon = \epsilon_1$ and $r = r_2$.
        Repeating in this manner, then multiplying the inequalities \eqref{eq:LFPPScalingConstantsStrategy} with $\epsilon = \epsilon_k$ and $r = r_{k+1}$ together and rearranging will yield the theorem.

        For $n \geq 1$, let $\zeta_{+}(n) \coloneqq 1 - \frac{1}{3^{n+1}}$ and $\zeta(n) \coloneqq 1 - \frac{1}{3^{n}}$ so that $1 - \zeta_{+}(n) < \g (1 - \zeta(n))$ for all $n \geq 1$.
        Let, say, $\zeta_{+}(0) \coloneqq \frac{1}{3}$.
        Choose $N(\delta)$ as in the statement of Lemma \ref{lemma:GoodRatios} with $\delta/(1 + \delta)$ in place of $\delta$.
        Applying Lemma \ref{lemma:GoodRatios} as in \eqref{eq:GoodScales}, there is some $\tilde{\epsilon} = \tilde{\epsilon}(N) \in (0,1)$ such that for each $\epsilon \in (0, \tilde{\epsilon})$, there is some $1 \leq n \leq 3N$ such that
        \begin{align*}
          \cR_{\epsilon}^{(n)} \
          &\coloneqq \ \left\{r \in (\epsilon^{1 - \zeta(n)}, \epsilon^{1 - \zeta_{+}(n)}) \cap \{8^j\}_{j \in \ZZ} : \frac{1}{1+ \delta} < \ratios < 1 + \delta \right\}
        \end{align*}
        has size $\geq \frac{1}{3} \#((\epsilon^{1 - \zeta(n)}, \epsilon^{1 - \zeta_{+}(n)}) \cap \{8^j\}_{j \in \ZZ})$.

        Fix $\epsilon_0 \in (0, \tilde{\epsilon})$ and inductively define $\epsilon_k$ for $k \geq 1$ as follows.
        If $\epsilon_{k-1} \geq \tilde{\epsilon}$, then let $\epsilon_k \coloneqq \epsilon_{k-1}$.
        If $\epsilon_{k-1} \in (0, \tilde{\epsilon})$, then choose $1 \leq n_k \leq 3N$ such that $\# \cR_{\epsilon_{k-1}}^{(n_k)} \geq \frac{1}{3} \#((\epsilon_{k-1}^{1 - \zeta(n_k)}, \epsilon_{k-1}^{1 - \zeta_{+}(n_k)}) \cap \{8^j\}_{j \in \ZZ})$, fix $r_k \in \cR_{\epsilon_{k-1}}^{(n_k)}$, and put $\epsilon_k \coloneqq \epsilon_{k-1}/r_k$.
        Let $K \coloneqq \inf\{k \geq 1 : \epsilon_k \geq \tilde{\epsilon}\}$, so then for all $1 \leq k \leq K$,
        \begin{align*}
          \left(1 + \delta \right)^{-1} r_k^{\xi Q - 1} \
          &\leq \ \frac{\fa_{\epsilon_k}}{\fa_{\epsilon_{k-1}}} \
          \leq \ \left(1 + \delta \right) r_k^{\xi Q - 1}.
        \end{align*}
        Multiply over all $1 \leq k \leq K$, then rearrange to get
        \begin{align}
          \fa_{\epsilon_K} \left(1 + \delta \right)^{-K} \left(\prod_{k=1}^{K} r_k \right)^{1 - \xi Q} \
          &\leq \ \fa_{\epsilon_0} \
          \leq \ \fa_{\epsilon_K} \left(1 + \delta \right)^K \left(\prod_{k=1}^{K} r_k \right)^{1 - \xi Q}.
          \label{eq:InitialLFPPScalingConstantsBound}
        \end{align}

        Now note that for $1 \leq k \leq K$,
        \begin{align*}
          \epsilon_k \
          &> \ \epsilon_0^{\zeta_{+}(n_1) \cdots \zeta_{+}(n_k)} \
          > \ \epsilon_0^{\zeta_{+}(3N)^k}.
        \end{align*}
        So 
        \begin{align*}
          K \
          &\leq \ \inf\left\{k \geq : \epsilon_0^{\zeta_{+}(3N)^{k}} \geq \tilde{\epsilon} \right\} \
          \leq \ \frac{\log \log \epsilon_0^{-1} - \log \log \tilde{\epsilon}^{-1}}{\log \zeta_{+}(3N)^{-1}} + 1.
        \end{align*}
        Therefore,
        \begin{align}
          \left(1 + \delta \right)^K \
          &\leq \ \exp\left\{\log(1 + \delta) \left(\frac{\log \log \epsilon_0^{-1} - \log \log \tilde{\epsilon}^{-1}}{\log \zeta_{+}(3N)^{-1}} + 1 \right)\right\} \
          \leq \ C \left(\log \epsilon_0^{-1} \right)^{\frac{\log(1 + \delta)}{\log \zeta_{+}(3N)^{-1}}}, \label{eq:1PlusDeltaUpperBound}
        \end{align}
        where
        \begin{align*}
          C \
          &\coloneqq \ \exp\left\{-\frac{\log(1 + \delta) \log \log \tilde{\epsilon}^{-1}}{\log \zeta_{+}(3N)^{-1}} + \log(1 + \delta) \right\}
        \end{align*}
        is a $\delta$-dependent constant.
        Furthermore, since
        \begin{align*}
          \epsilon_K \
          &= \ \frac{\epsilon_{K-1}}{r_K} \
          = \ \frac{\epsilon_{K-2}}{r_{K-1} r_K} \
          = \ \cdots \
          = \ \frac{\epsilon_0}{r_1 r_2 \cdots r_K},
        \end{align*}
        and also
        \begin{align}
          \tilde{\epsilon} \
          &\leq \ \epsilon_K \
          \leq \ \epsilon_{K-1}^{\zeta(n_{K-1})} \
          < \ \tilde{\epsilon}^{\zeta(n_{K-1})} \
          < \ 1,
          \label{eq:KthStepBound}
        \end{align}
        it follows that
        \begin{align}
          \left(\prod_{k=1}^{K} r_k \right)^{1 - \xi Q} \
          &= \ \left(\frac{\epsilon_0}{\epsilon_K} \right)^{1 - \xi Q} \
          \in \ \left[\epsilon_0^{1 - \xi Q} , \tilde{\epsilon}^{\xi Q - 1} \epsilon_0^{1 - \xi Q} \right].
          \label{eq:ProductOfRadii}
        \end{align}
        Combining \eqref{eq:InitialLFPPScalingConstantsBound}, \eqref{eq:1PlusDeltaUpperBound}, and \eqref{eq:ProductOfRadii}, we arrive at
        \begin{align*}
          \fa_{\epsilon_K} C^{-1} \left(\log \epsilon_0^{-1} \right)^{-\frac{\log(1 + \delta)}{\log \zeta_{+}(3N)^{-1}}} \epsilon_0^{1 - \xi Q} \
          &\leq \ \fa_{\epsilon_0} \
          \leq \ \fa_{\epsilon_K} C \left(\log \epsilon_0^{-1} \right)^{\frac{\log(1 + \delta)}{\log \zeta_{+}(3N)^{-1}}} \epsilon_0^{1 - \xi Q} \tilde{\epsilon}^{\xi Q - 1}.
        \end{align*}
        Note that $\frac{\log(1 + \delta)}{\log \zeta_{+}(3N)^{-1}} \leq b$.
        Since $\tilde{\epsilon}$ depends only on $\delta = e^b - 1$, $\epsilon \mapsto \fa_{\epsilon}$ is bounded above and below by positive $(\xi, b)$-dependent constants on $[\tilde{\epsilon},1]$.
        Since $\epsilon_K \in [\tilde{\epsilon}, 1]$ by \eqref{eq:KthStepBound}, we get
        \begin{align*}
          \tilde{C}^{-1} \left(\log \epsilon_0^{-1} \right)^{-b} \epsilon_0^{1 - \xi Q} \
          &\leq \ \fa_{\epsilon_0} \
          \leq \ \tilde{C} \left(\log \epsilon_0^{-1} \right)^b \epsilon_0^{1 - \xi Q}
        \end{align*}
        for some $(\xi,b)$-dependent constant $\tilde{C} > 0$.
      \end{proof}

    \subsection{LQG Scaling Formula}
      We will conclude by using the exact scaling formula \eqref{eq:LFPPScaling} to deduce Theorem \ref{thm:LQGScaling}.
  
      \begin{proof}[Proof of Theorem \ref{thm:LQGScaling}]
        By \eqref{eq:LFPPScaling},
        \begin{align}
          \LFPP(az+b, aw+b) \
          &= \ \frac{|a| \fa_{\epsilon}^{-1}}{|a|^{\xi Q} \fa_{\epsilon/|a|}^{-1}} \LFPP[\epsilon/|a|][h(a \cdot + b) + Q \log |a|](z,w).
          \label{eq:LFPPScalingRedux}
        \end{align}
        Now let $A$ be the almost sure event that $\LFPP \to D_h$ locally uniformly.
        Corollary 1.11 from \cite{ExistenceAndUniqueness} together with \eqref{eq:LFPPScalingRedux} implies that on $A$, $\LFPP[\epsilon/|a|][h(a \cdot + b) + Q \log |a|]$ also converges locally uniformly with limit $D_h(a z + b, aw + b)$ for \textit{every} $a \in \CC \setminus \{0\}$ and $b \in \CC$.
        So we can define $D_{h(a \cdot + b) + Q \log|a|}$ to equal $D_h(a \cdot + b, a \cdot + b)$ on $A$, and define it arbitrarily on $A^c$.
      \end{proof}

  \clearpage
  \nocite{*} 
  \printbibliography
\end{document}